\documentclass[12pt]{amsart}

\usepackage[dvips]{color}
\usepackage{amscd}
\theoremstyle{plain}
\newtheorem{thm}{Theorem}[section]
\newtheorem*{thm*}{Theorem}
\newtheorem{lemma}{Lemma}[section]
\newtheorem{definition}{Definition}[section]
\newtheorem{prop}{Proposition}
\newtheorem{cor}{Corollary}[section]
\usepackage{graphicx}
\usepackage{latexsym}
\usepackage{amssymb}
\usepackage{amsmath}
\usepackage{amsthm}
\theoremstyle{remark}

\newcommand{\cA}{\mathcal{A}}

\newcommand{\cE}{\mathcal{E}}

\newcommand{\cP}{\mathcal{P}}

\newcommand{\C}{\mathbb{C}}
\newcommand{\R}{\mathbb{R}}
\newcommand{\MF}{\mathcal{MF}}
\newcommand{\Q}{\mathbb{Q}}
\newcommand{\Z}{\mathbb{Z}}

\newcommand{\PMF}{\mathcal{PMF}}
\newcommand{\M}{\mathcal{M}}

\newcommand{\bz}{\boldsymbol{z}}
\newcommand{\HH}{{\mathbb H}}
\DeclareMathOperator{\diag}{diag}
\DeclareMathOperator{\diam}{diam}

\DeclareMathOperator{\SL}{SL}

\newcommand{\vhi}{\varphi}
\newcommand{\eps}{\varepsilon}

\title[Winning games for bounded geodesics]{Winning games for bounded geodesics in moduli spaces of quadratic differentials}
\begin{document}
\author{Jonathan Chaika}
\author{Yitwah Cheung}
\author{Howard Masur}
\thanks{JC: Partially supported by NSF grant DMS 1004372.}
\thanks{YC: Partially supported by NSF grant DMS 0956209.}
\thanks{HM: Partially supported by NSF grant DMS 0905907.}
\address{Mathematics Department, 
155 S 1400 E Room 233, 
University of Utah, 
Salt Lake City, UT 84112-0090, USA}
\email{chaika@math.utah.edu}
\address{Mathematics Department,
San Francisco State University,
1600 Holloway Avenue,
San Francisco, CA 94132, USA}
\email{ycheung@sfsu.edu}
\address{Department of Mathematics,
University of Chicago,
5734 S. University,
Chicago, IL 60637, USA}
\email{masur@math.uchicago.edu}

\begin{abstract}
We prove that the set of bounded geodesics in Teichm\"{u}ller space are a winning set 
for Schmidt's game. This is a notion of largeness in a metric space that can apply 
to measure $0$ and meager sets. We prove analogous closely related results on 
any Riemann surface, in any stratum of quadratic differentials, on any Teichm\"{u}ller 
disc and for intervals exchanges with any fixed irreducible permutation.
\end{abstract}

\maketitle

\section{Introduction}
In the 1966 paper \cite{Sgame} W.~Schmidt introduced a game, now called a 
Schmidt game, to be played by two players in $\R^n$.  He showed that winning 
sets for his game are large in the sense that they have full Hausdorff dimension, 
and that the set of badly approximable vectors in $\R^n$, which were known to 
have measure zero, is a winning set for this game.  Schmidt's game and a 
modified version of it were used in \cite{Da86} and \cite{KW10} to establish that 
the set of bounded trajectories of nonquasiunipotent flow on a finite volume 
homogeneous space has full Hausdorff dimension, a result first established 
in \cite{KM96} using different methods.  The dynamical significance of badly 
approximable vectors is well-understood: in terms of the flow on the moduli 
space of $(n+1)$-dimensional tori $\SL(n+1,\R)/\SL(n,\Z)$ induced by the 
left action of the one-parameter subgroup $g_t=\diag(e^t,\dots,e^t,e^{-nt})$, 
a vector $\mathbf{x}\in\R^n$ is badly approximable if and only if it determines 
a bounded trajectory via $g_tU(\mathbf{x})\SL(n+1,\Z)$ where $U(\mathbf{x})$ 
is the unipotent matrix whose $(i,j)$-entry is $1$ if $i=j$, $-x_i$ if $i\le n$ and 
$j=n+1$, and $0$ otherwise.  

This paper is concerned with higher genus analogues of the same circle of ideas.  
Let $\PMF$ be Thurston's sphere of projective measured foliations on a closed 
surface of genus $g>1$.  
Let $D\subset \PMF$ consist of those foliations $F$ such that for some (hence all) 
quadratic differentials $q$ whose vertical foliation is $F$, the Teichm\"uller geodesic 
defined by $q$ stays in a compact set in $\M_g$, the moduli space of genus $g$.  
Following \cite{McM2}, we use the terminology {\em Diophantine} to describe the 
foliations that lie in the set $D$.  It was shown in \cite{McM2} that a foliation $F$ is 
Diophantine if and only if $$\inf_{\beta} |\beta|i(F,\beta)>0.$$ 
Here, the infimum is over all homotopy classes of simple closed curves $\beta$, 
$i(F,\beta)$ is the standard intersection number, and for $|\beta|$ one can take it 
to be the length of the geodesic in the homotopy class with respect to some fixed 
hyperbolic structure on the surface.  The notion of Diophantine does not depend 
on which hyperbolic metric is chosen.  Alternatively, we may fix a triangulation of 
the surface and take the number $|\beta|$ to be the minimum number of edges 
traversed by any curve in the homotopy class.  
In the moduli space of genus one, a.k.a. the modular surface, Teichm\"uller 
geodesic rays are represented by arcs of circles or vertical lines in the upper 
half plane with endpoint in $\R\cup\{\infty\}$.  The notion of Diophantine extends 
the property of a geodesic ray that its endpoint is a badly approximable real number.  

In \cite{McM} McMullen introduced two variants of the Schmidt game, giving 
rise to the notions of \emph{strong winning} and \emph{absolute winning} sets.  
(We recall their definitions in the next section.)  It is not hard to show that 
absolute winning implies strong winning while strong winning implies 
\emph{Schmidt winning}, i.e. winning in the original sense of Schmidt.  
In particular, they also have full Hausdorff dimension.  
McMullen raises the question in \cite{McM} as to whether the set of Diophantine 
foliations $D\subset\PMF$ is a strong winning set.  
In this paper, we give an affirmative answer to this question.  

\begin{thm}\label{thm:pmf}
The set $D\subset\PMF$ of Diophantine foliations is a strong winning set, 
hence a winning set for the Schmidt game.  However, it is not absolute winning.  
\end{thm}

We remark that the Schmidt game requires a metric on the space, 
but the notion of winning is invariant under bi-Lipschitz equivalence.  
In the case of $\PMF$ there are various bi-Lipschitz equivalent ways of 
defining a metric. The most familiar is to use train track coordinates.  
(See \cite{PH} for a discussion of train tracks).  A fixed train track 
defines a local metric by pull-back of the Euclidean metric.  
A finite collection of train tracks can be used to parametrize all of $\PMF$.  
One can then defined a path metric on $\PMF$ via a finite number of 
locally defined metrics.   

The next theorem concerns quadratic differentials (see Section \ref{S:qd} for the definition of quadratic differentials)
 that determine 
bounded geodesics in a stratum. 
\begin{thm}\label{thm:fullspace}
Let $Q^1(k_1,\ldots, k_n,\pm)$ be any stratum of unit norm quadratic differentials.  
Let $U$ be an open set with compact closure $\bar U \subset Q^1(k_1,\ldots, k_n,\pm)$ with a metric given by the pull-back of the Euclidean metric under a local coordinate 
system given by the holonomy coordinates of saddle connections.   Then there exists 
an $\alpha>0$ depending on the smallest systole in $U$ such that the subset 
$E_Q\subset \bar U$ consisting of those quadratic differentials $q$ such that 
the Teichm\"uller geodesic defined by $q$ stays in a compact set in the stratum 
is an $\alpha$-strong winning, hence winning for the Schmidt game.  
It is not absolute winning.  
\end{thm}

Again we remark that the metric is not canonical as it depends on a choice of 
coordinates.  However different choices give bi-Lipschitz equivalent metrics 
and the notion of winning is well-defined.  We remark that bounded has a 
slightly more restrictive meaning here than in the case of $\PMF$ in that in this case 
no saddle connection gets short along the geodesic, while in the case of $\PMF$ 
the condition is slightly weaker in that no simple closed curve gets short.  
The difference in definitions is accounted for by the fact that points in $\PMF$ 
are only defined up to equivalence by Whitehead moves (\cite{FLP}) which 
collapse leaves of a foliation joining singularities to a higher order singularity.  
Thus quadratic differentials whose vertical foliations determine the same point 
in $\PMF$ may lie in different strata.  

\begin{thm}\label{thm:riemann}
Fix a closed Riemann surface $X$ of genus $g>1$ and let $Q^1(X)$ denote 
the space of unit norm holomorphic quadratic differentials on $X$.  
Then the set of $q\in Q^1(X)$ that determine a Teichm\"uller geodesic that 
stays in a compact set of the stratum is strong winning, hence Schmidt winning.  
It is not absolute winning.
\end{thm}

Here the distance is defined by the norm; namely $d(q_1,q_2)=\|q_1-q_2\|$.
\begin{thm}\label{thm:iet}
Let $\Lambda$ denote the simplex of interval exchange transformations 
$(T,\lambda, \pi)$ on $n$ intervals with a fixed irreducible permutation $\pi$ 
defined on the unit interval $[0,1)$.  We give $\Lambda$ the Euclidean metric.  
Let $E_B$ consist of the {\em bounded}  $(T,\lambda,\pi)$.  This means that 
$\inf_n n|T^n(p_1)-p_2|>0$, where $p_1,p_2$ are discontinuities of $T$.   
Then $E_B$ is strong winning hence Schmidt winning.  It is not absolute winning.
\end{thm}

Because winning has nice intersection properties we obtain the following result 
that there are many interval exchange transformations which are bounded and 
any reordering of the lengths is also bounded.
\begin{cor}
Let $E_B$ be as in Theorem~\ref{thm:iet}.  Then the set 
$$\{\lambda \in \Lambda:(\lambda_{i_1},...\lambda_{i_n})\in E_B 
  \text{ for all } \{i_1,...,i_n\}=\{1,...,n\}\}$$ is nonempty. 
In fact, it has full Hausdorff dimension.
\end{cor}

The main theorem we prove from which the other theorems will follow is 
a one-dimensional version. 
\begin{thm}\label{main}
Let $q$ be a holomorphic quadratic differential on a closed Riemann surface of 
genus $g>1$.  Then the set $E$ of directions $\theta$ in the circle $S^1$ with 
the Euclidean metric such that the Teichm\"uller geodesic defined by $e^{i\theta} q$ 
stays in a compact set of the corresponding stratum in the moduli space of 
quadratic differentials is an absolute winning set; hence strong winning.   
\end{thm}

In Theorem~\ref{main} we call the directions in $E$ {\em bounded} directions.
Here is an equivalent formulation of Theorem \ref{main} (Proposition~\ref{equiv cond} establishes the equivalence). 

\begin{thm}\label{thm:equiv} Let $$S=\{(\theta,L): \theta \text{ is the direction of a saddle connection of $q$ of length } L\}.$$
 Then the set $E$ of bounded directions $\psi$ in the circle is the same as  
 $$\{\psi:\underset{(\theta,L)\in S}{\inf}L^2|\theta-\psi|>0\}$$ 
 and this is an absolute winning set.  
\end{thm}

As an immediate corollary we get the following result which was first proved by Kleinbock and Weiss \cite{baflat} using quantitative non-divergence of horocycles \cite{MW}.
\begin{cor}
The set of directions such that the Teichm\"uller geodesic stays in a compact subset of the stratum has Hausdorff dimension $1$.
\end{cor}

It is well-known that a billiard in a  polygon $\Delta$ whose vertex angles are 
rational multiples of $\pi$ gives rise to a translation surface by an unfolding process.  
We have the following  corollary to Theorem~\ref{thm:equiv}.  
\begin{cor}\label{cor:rational}
Let $\Delta$ be a rational polygon.  The set $E$ of directions $\theta$ 
for the billiard flow in $\Delta$ with the property that  there is 
an $\epsilon=\epsilon(\theta)>0$, so that for all $L>0$, the billiard path 
in direction $\theta$ of length smaller than $L$ starting at any vertex, 
stays outside an $\frac{\epsilon}{L}$ neighborhood of all vertices of $\Delta$, 
is an absolute winning set.  
\end{cor}

Since for any $0<\alpha<2\pi$ and $n\in\Z$, the set $E+n\alpha$ (mod $2\pi$) is 
an isometric image of $E$, it is also Schmidt winning with the same winning constant 
for each $n$.  Therefore, the infinite intersection $\cap_{n\in\Z} (E+n\alpha)$ is 
Schmidt winning, and thus has Hausdorff dimension $1$.  
This gives the following corollary.
\begin{cor}
For any $\alpha$ there is a Hausdorff dimension $1$ set of angles $\theta$ 
such that for any $n$, there is $\epsilon_{n}>0$ such that a billiard path at 
angle $\theta+n\alpha (mod(2\pi))$ and length at most $L$ from a vertex 
does not enter a neighborhood of radius $\frac{\epsilon_n}{L}$ of any vertex. 
\end{cor}

Another corollary uses the absolute winning property but does not follow 
just from Schmidt winning.  Let $E$ be the set from Corollary~\ref{cor:rational}. 
\begin{cor}
The set $E'=\{\theta:\forall n\in\Z_{>0}, n\theta\in E\}$ 
has Hausdorff dimension $1$.
\end{cor}

The core theorems that we prove are Theorem~\ref{main} for the set of bounded 
directions in the disc and Theorem~\ref{thm:fullspace} for winning in the stratum.  
The former is a model for the latter although the former proves absolute winning and 
the latter strong winning.  The fairly general Theorem~\ref{abstract thm} will reduce 
Theorem~\ref{thm:pmf} and Theorem~\ref{thm:iet} to Theorem~\ref{thm:fullspace}.  
Theorem~\ref{thm:riemann} reduces to Theorem~\ref{thm:pmf}.

\subsection{Acknowledgments}
The authors would like to thank Dmitry Kleinbock and Barak Weiss for telling the 
first author about this delightful problem and helpful conversations and to thank 
Curtis McMullen for his conversations with the third author.  The authors would 
also like to thank the referee for numerous helpful suggestions.

\section{Strong and absolute winning sets}\label{S:winning}

\subsection{Schmidt games}
We describe the Schmidt game in $\R^n$.  Suppose we are given a set $E\subset\R^n$.
Suppose two players Bob and Alice take turns choosing a sequence of closed Euclidean balls  $$B_1\supset A_1\supset B_2\supset A_2\supset B_3\ldots $$ (Bob choosing the $B_i$ and Alice the $A_i$) whose diameters satisfy, for fixed $0<\alpha,\beta<1$,  
\begin{equation}\label{def:Schmidt}
   |A_i|=\alpha |B_i| \quad\text{and}\quad |B_{i+1}|=\beta |A_i|.  
\end{equation}
Following Schmidt 
\begin{definition}
We say $E$ is an $(\alpha,\beta)$-{\em winning set} if Alice has a strategy so that no matter what Bob does, $\cap_{i=1}^\infty B_i\in E$. It is $\alpha$-{\em winning} if it is $(\alpha,\beta)$-winning for all $0<\beta<1$.  $E$ is a {\em winning set for Schmidt game} if it is $\alpha$-winning for some $\alpha>0$.
\end{definition}

Their main properties, proved by Schmidt in \cite{Sgame}, are: 
\begin{itemize}
  \item they have full Hausdorff dimension,
  \item they are preserved by bi-Lipschitz mappings (the constant $\alpha$ can change), 
  \item a countable intersection of $\alpha$-winning sets is $\alpha$-winning. 
\end{itemize}

McMullen \cite{McM} suggested two variants of the Schmidt game as follows.  
The first variant replaces (\ref{def:Schmidt}) with 
\begin{equation}\label{def:strong}
   |A_i|\ge\alpha |B_i|\quad\text{and}\quad |B_{i+1}|\ge\beta |A_i|.  
\end{equation}
The notions of $(\alpha,\beta)$-strong and $\alpha$-strong winning sets are similarly defined.  (Bob wins if $\cap_{i=1}^\infty B_i\cap E=\emptyset$; otherwise, Alice wins).  A \emph{strong winning set} refers to a set that is $\alpha$-strong winning for some $\alpha>0$.  

In the second variant, the sequence of balls $B_i$, $A_i$ must be chosen so that 
  $$B_1\supset  B_1\setminus A_1\supset B_2\supset B_2\setminus A_2
     \supset B_3 \supset\dots $$ and for some fixed $0<\beta<1/3$,  
  $$|A_i|\leq \beta |B_i| \quad\text{and}\quad |B_{i+1}|\geq \beta |B_i|.$$  
We say $E$ is $\beta$-\emph{absolute winning} if Alice has a strategy that forces 
  $\cap_{i=1}^\infty B_i\cap E=\emptyset$ regardless of how Bob responds.  
An \emph{absolute winning set} is one that is $\beta$-absolute winning for all $0<\beta<1/3$.  (Remark: The condition $\beta<1/3$ ensures Bob always has moves 
available to him no matter how Alice plays her moves.)  
It is also clear that if a set is absolute winning for some $\beta_0$ then it is absolute winning for $\beta>\beta_0$.

As noted in the introduction, absolute winning implies strong winning, which in turn 
implies winning in the sense of Schmidt.  In particular, both types of sets have 
full Hausdorff dimension.  These notions provide two new classes of sets that also 
have the countable intersection property and are not only bi-Lipschitz invariant, 
but preserved by the much larger class of quasi-symmetric homeomorphisms.  
(See \cite{McM}.)  
As McMullen notes, most sets known to be winning in the sense of Schmidt 
are in fact strong winning, as is the case with the set of badly approximable 
vectors in $\R^n$.  Since any subset of $\R^n$ that contains a line segment 
in its complement cannot be absolute winning (because Bob can always 
choose $B_j$ centered at a point on this line segment), the set of badly 
approximable vectors in $\R^n$ ($n\ge2$) provides a natural example of 
a strong winning set that is not absolutely winning.  However, it is far from 
obvious that there are winning sets in the sense of Schmidt that are not 
strong winning (\cite{McM}).  

\subsection{Projections and the simultaneous blocking game}
Theorem~\ref{thm:pmf} and Theorem~\ref{thm:iet} will follow from  Theorem~\ref{thm:fullspace} by use of 
the following fairly general statement. 

\begin{definition}
A surjective map $f:X\to Y$ between metric spaces is a quasi-symmetry
if there exists $0<c<1$ such that for all $(x,r)\in X\times \R_+$,
$$B_Y(f(x),cr)\subset  f(B_X(x,r))\subset B_Y(f(x),r/c).$$

\end{definition}

\begin{thm} \label{abstract thm}
Suppose $f:X\to Y$ is a surjective quasi-symmetry between complete metric spaces and $E\subset X$ is $\alpha$-strong winning.
Then  $f(E)$   is $c^2\alpha$-strong winning. 
(Here winning means that once Bob chooses an initial  ball in $Y$ then Alice has a strategy to force the intersection point to lie in $f(E)$).
\end{thm}

We remark that linear projection  maps from $\mathbb{R}^n$ onto subspaces obviously satisfy the hypotheses and these are what will occur in the proofs of Theorem~\ref{thm:pmf} and Theorem~\ref{thm:iet}, but we wish to prove a more general theorem. 

\begin{proof}
First we claim  that if $s<r, r_0',x_0'\in X,z\in Y$, and $y'\in B_X(x_0',r_0')$ are such that   $B_Y(z,s)\subset B_Y(f(y'),cr)$ then  there exists $z'\in f^{-1}(z)$ such
that $B_X(z',s)\subset B_X(y',r)$.
We prove the claim. 

Since $B_Y(z,s)\subset B_Y(f(y'),cr)$ we have  
$B_Y(z,cs)\subset B_Y(f(y'),cr)$ which in turn implies  that 
$z\in B_Y(f(y'),c(r-s))$. 
It follows  from the hypotheses of the Theorem that there is a $z'\in B_X(y',r-s)$ such that $f(z')=z$.
The triangle inequality now implies that $B_X(z',s)\subset B_X(y',r)$, proving the claim. 

Now we  show that $f(E)$ is $(c^2\alpha,\beta)$-strong winning in $Y$ by winning an auxiliary $(\alpha, c^2\beta)$-strong winning game in $X$. 
We  describe the inductive strategy.  We are given  $B_X(x_k,r)$ and $B_Y(f(x_k),cr)$ where $B_X(x_k,r)$ is part of Alice's $(\alpha, c^2 \beta)$-strong winning strategy in $X$ and $B_Y(f(x_k),cr)$ is Alice's move in the $(c^2\alpha,\beta)$-game in $Y$.
 Bob chooses $B_Y(z,s) \subset B_Y(f(x_k),cr)$ and $s\geq \beta cr$. 

 By the claim  there exists $z'\in f^{-1}(z)$  such that 
$$B_X(z',cs) \subset B_X(x_k,r) \text{ with }cs\geq c^2\beta r.$$ 
So $B_X(z',cs)$ is a legal move in the $(\alpha,c^2\beta)$-strong winning game (in $X$) given Alice's move $B_X(x_k,r)$. 
So Alice has a response $$B_X(x_{k+1}, t) \subset B_X(z',cs)\text{ and }t\geq \alpha cs$$
as part of her $(\alpha, c^2\beta)$-strong winning strategy in $X$, which we assumed existed.

 Now 
 $$B_Y(f(x_{k+1}),ct) \subset f(B_X(x_{k+1},t))\subset f(B_X(z',cs)) \subset B_Y(z,s) \text{ and }ct\geq c^2\alpha s.$$ So it is a legal move for the $(c^2\alpha,\beta)$-strong winning game given Bob's move $B_Y(z,s).$
 
 Because the auxiliary game is a legal $(\alpha,c^2\beta)$-game we have $\cap_{k=1}^{\infty} B_X(x_k,r_k) \in E$. Thus 
 $$\cap_{k=1}^{\infty} B_Y(f(x_k),cr) \subset \cap_{k=1}^{\infty} f(B_X(x_k,r)) \in f(E).$$ So $f(E)$ is $(c^2\alpha,\beta)$-strong winning.
 \end{proof}
 
To prove Theorem~\ref{main}, (played on the circle $S^1$) it will be convenient to consider a variation on the absolute winning game where Alice is permitted to simultaneously block $M$ intervals of radii $\le\beta|I_j|$.\footnote{The authors would like to thank Barak Weiss 
for bringing our attention to this variant of the absolute winning game.}  (The condition 
$(2M+1)\beta<1$ ensures that Bob will always have available moves.)  

\begin{lemma}\label{lem:simblock}
Suppose $E$ is a winning set for the modified game where Alice is permitted to 
simultaneously block $M$ intervals of length at most $\beta^M$ times the length 
of Bob's interval.    Then $E$ is an absolute winning set with the parameter $\beta$.  
\end{lemma}
\begin{proof}
Let $I_j,j\ge1$ denote the intervals that Bob plays in the (original) absolute game.  
Alice will consider the subsequence $I_{1+rM},r\ge0$ as Bob's moves in the modified game.  
Given $j=1\mod M$ Alice considers the intervals $J_1,\dots,J_M$ she would have played in 
  the modified game in response to Bob's choice of $I_j$.  
The strategy for her next $M$ moves of the original game is to pick 
  $$U_j=J_1, \quad U_{j+1}=J_2\cap I_{j+1}, \quad U_{j+2}=J_3\cap I_{j+2}, 
        \quad \dots \quad U_{j+M-1}=J_M\cap I_{j+M-1}.$$
Observe that for $i=1,\dots,k-1,$ 
  $$|U_{j+k-1}| \le |J_k| \le \beta^M |I_j| \le \beta^{M-i}|I_{j+i}| \le \beta |I_{j+k-1}|$$
  so Alice's choice of $U_{j+k-1}$ in response to $I_{j+k-1}$ is valid.  
Note that $|I_{j+M}|\ge\beta^M|I_j|$ and that $I_{j+M}$ is disjoint from $J_k$ because 
  $I_{j+M}\subset I_{j+k}$ and Bob is required to choose $I_{j+k}$ disjoint from $J_k$ 
  inside $I_{j+k-1}$.  
Thus, $I_{j+M}$ is a valid move for Bob in the modified game so that Alice can 
  continue her next $M$ moves by repeating the strategy just described.  

Since the intervals $I_j$ are nested, we have 
  $$\bigcap_{j=1}^\infty I_j = \bigcap_{r=1}^\infty I_{1+rM},$$ 
  which has nontrivial intersection with $E$, by hypothesis.  
\end{proof}

We remark that there is an obvious partial converse: \emph{if Alice can win the absolute 
game then she can also win the modified game with the same parameter}.  Indeed, she 
simply picks all her $M$ intervals to be the same as the interval she would have chosen 
in the original, absolute game.

\subsection{Case of badly approximable numbers}
Recall that a real number $\theta$ is badly approximable if there exists $c>0$ such that for all rationals $p/q\in\Q$  $$ \left|\theta - \frac{p}{q} \right| > \frac{c}{q^2}.$$  
The fact that the set of badly approximable numbers is absolute winning is a special case of Theorem 1.3 of \cite{McM}.  We give a proof of this result because it serves as a motivation for the proof of Theorem~\ref{main}.  
\begin{thm}\label{thm:BA}
The set of badly approximable real numbers is absolute winning.  
\end{thm}
\begin{proof}
Fix $\eps>0$.  
Given an interval $B_j$ chosen by Bob, let $I_j$ be an $\eps|B_j|$-neighborhood 
of $B_j$ and let $p_j/q_j$ be the rational of smallest denominator (in lowest terms) 
in the interval $I_j$.  Alice's strategy is to "block $p_j/q_j$"; in other words, she picks 
$$A_j = \left(\frac{p_j}{q_j}-\frac{\beta|B_j|}{2},\frac{p_j}{q_j}+\frac{\beta|B_j|}{2}\right).$$

We claim that there exists $c>0$ such that $|I_j|q_j^2 >c$ for all $j\ge1$.  
Indeed, if $q_{j+1}<\beta^{-1}q_j$ then 
\begin{equation}\label{ieq:sep}
  |I_j| > \left|\frac{p_j}{q_j} - \frac{p_{j+1}}{q_{j+1}}\right| \ge \frac{1}{q_jq_{j+1}} 
            > \frac{\beta}{q_j^2}
\end{equation}
   whereas if $q_{j+1}\ge\beta^{-1}q_j$ we have 
  $$|I_{j+1}|q_{j+1}^2 \ge \beta|I_j|q_{j+1}^2 \ge \beta^{-1}|I_j|q_j^2.$$  
Now whenever  the quantity $|I_j|q_j^2$ is less than $\beta$, the above inequality says it  must increase by a factor 
of at least $\beta^{-1}$ at each step until it exceeds $\beta$, after which it may decrease 
by a factor of at most $\beta$ (since $|I_{j+1}|\ge\beta|I_j|$ and $q_{j+1}\ge q_j$) 
and then exceed $\beta$ again at the following step.  
Hence, $\liminf |I_j|q_j^2>\beta^2$ and the claim follows. 

Given $x\in\cap B_j$ and $p/q\in\Q$, suppose first that $p/q\not\in I_1$.  
Then $$\left|x-\frac{p}{q}\right| \ge \eps|B_1| \ge \frac{\eps|B_1|}{q^2}.$$  
Thus assume $p/q\in I_1$. Since our strategy guarantees that $p/q\notin \cap I_j$,  
there is a (unique) index $j$ such that $p/q\in I_j\setminus I_{j+1}$ 
  and $q_j\le q$ (because $p/q\in I_j$) and since $p/q\not\in I_{j+1}$ we have 
  $$\left|x-\frac{p}{q}\right| \ge \eps|B_{j+1}| \ge \beta\eps|B_j| > \frac{c\beta\eps|I_j|}{1+2\eps}>\frac{c\beta\eps}{(1+2\eps)q_j^2} \ge \frac{c\beta\eps}{(1+2\eps)q^2}$$  
proving $x$ is badly approximable.  
\end{proof}

\subsection{Sketch of the proofs of Theorem~\ref{main} and Theorem~\ref{thm:fullspace}}
In the game played with quadratic differentials on a higher genus surface, 
we have a similar criterion as for the torus, given by Proposition~\ref{equiv:cond} 
that for a Teichm\"uller geodesic to lie in a compact set in the stratum, 
the direction is far from the direction of a saddle connection.  This says that 
in order to show this set is winning we want to find a strategy giving us a point 
far from the direction of any saddle connection. Unlike the genus one case 
we have the  major complication that directions of saddle connections in general 
need not be separated in the sense that the angle between them need {\em not} 
be at least a constant over the product of their lengths as it is in the case of tori.  
Equivalently, it may happen that on some flat surfaces there are many intersecting 
short saddle connections.  This forces us to consider {\em complexes} of saddle 
connections that become simultaneously short under the geodesic flow.  
We call these complexes {\em shrinkable}.  

An important tool is a process of combining a pair of  shrinkable complexes of 
a certain level or complexity  to build a  shrinkable complex of higher level.   This 
is given by Lemma~\ref{lem:combine} with the preliminary Lemma~\ref{lem:sigma}.  
These ideas are not really new; having appeared in several papers beginning with 
\cite{KMS}. 

The main point in this paper and the strategy is given by Theorem~\ref{blocking strategy}.  
We show first that complexes of highest level are separated, as in the case of the 
torus, for otherwise we could combine them to build a complex of higher level that 
is shrinkable.  This is impossible by definition of highest level.  We develop a strategy 
for Alice where, as in the torus case, she blocks these highest level complexes.  
Then we consider complexes of one lower level that lie in the complement of 
the interval used to block highest level intervals and which are not too long in 
a certain sense depending on the stage of the game.  We show that these are 
separated as well, for if not, we could combine them into a highest level complex 
of bounded size and these have supposedly been blocked at an earlier stage of 
the game.  Thus there can be at most one such lower level complex (up to a certain 
combinatorial equivalence) and we block it.  We continue this process inductively 
considering complexes of decreasing level one step at a time, ending by blocking 
single saddle connections.  Then after a fixed number of steps we return to blocking 
highest level complexes and so forth.  From this strategy, Theorem~\ref{main} 
will follow.  For technical reasons, we need to block complexes by intervals whose 
length is comparable to the reciprocal of the product of their longest saddle connection 
and the longest saddle connection on their boundary.  

In adapting the argument to the proof of Theorem~\ref{thm:fullspace}, we need 
to consider complexes on distinct flat surfaces.  In order to combine them so that 
Lemma~\ref{lem:combine} can be applied, we need to consider the problem of 
moving a complex on one surface to a nearby one.  (See Theorem~\ref{thm:move}.) 
This operation is not canonical since unlike parallel transport, it does not respect 
the operation of concatenation along paths.   However it does preserve inclusion of 
complexes (Proposition~\ref{prop:inclusion}) and this is sufficient for our purposes.  
While the basic strategy is the same as that in the proof of Theorem~\ref{main}, 
we caution the reader that unlike the ordinary and strong winning games, after Alice chooses $A_j$, the game "continues" in $B_j\setminus A_j$ rather than inside $A_j$.  
In particular, we do not have an analog of the simultaneous blocking strategy 
Lemma~\ref{lem:simblock}.

\section{Quadratic differentials, complexes, and geodesic flow}\label{S:qd}
\subsection{Quadratic differentials}
A general reference here is \cite{MaT}.
Recall a holomorphic quadratic differential $q=\phi(z)dz^2$ on a Riemann surface $X$ 
of genus $g>1$ defines for each local holomorphic coordinate $z$, a holomorphic 
function $\phi_z(z)$ such that in overlapping coordinate neighborhoods $w=w(z)$ 
we have $$\phi_w(w)\left(\frac{dw}{dz}\right)^2=\phi_z(z).$$  
On a compact surface, $q$ has a finite set $\Sigma$ of zeroes.  On the complement 
of $\Sigma$ there are {\em natural} local coordinates $z$ such that $\phi_z(z)\equiv 1$ 
and hence $q$ defines a flat surface.  A zero of order $k$ defines a cone singularity 
of angle $(k+2)\pi$.  Suppose $q$ has zeroes of orders $k_1,\ldots, k_p$ with $\sum k_i=4g-4$.  There is a moduli space or stratum $Q=Q(k_1,\ldots, k_p,\pm)$ of 
quadratic differentials all of which have zeroes of orders $k_i$.  The $+$ sign occurs 
if $q$ is the square of an Abelian differential and the $-$ sign otherwise. 

A quadratic differential $q$ defines an area form $|\phi(z)||dz^2|$ and a metric 
$|\phi|^{1/2}|dz|$.  We assume that our quadratic differentials have area one.  
Recall a {\em saddle connection} is a geodesic in the metric joining a pair of 
zeroes which has no zeroes in its interior. 
By the {\em systole} of $q$ we mean the length of the shortest saddle connection. 

A choice of a branch of $\phi^{1/2}(z)$ along a saddle connection $\beta$ and 
an orientation of $\beta$ determines a holonomy vector 
  $$hol(\beta)=\int_\beta \phi^{1/2}dz\in\C.$$  It is defined up to sign.  
Thinking of this as a vector in $\mathbb{R}^2$ gives us the horizontal and vertical 
components defined up to sign.  We will denote by $h(\gamma)$ and $v(\gamma)$ 
the absolute value of these components.  We will denote its length $|\gamma|$ as 
the maximum of $h(\gamma)$ and $v(\gamma)$.  This slightly different definition 
will cause no difficulties in the sequel.  

Given $\epsilon>0$, let $Q^1_{\epsilon}$ denote the compact set of unit area 
quadratic differentials in the stratum such that the shortest saddle connection 
has length at least $\epsilon$.  The group $SL(2,\mathbb{R})$ acts on $Q^1$ 
and on saddle connections.  (In the action we will suppress the underlying 
Riemann surface).  Let  
  $$g_t=\left( \begin{array}{cc}e^{t} & 0\\0& e^{-t}\end{array}\right)$$ 
  denote the Teichm\"uller flow acting on $Q^1$ and 
  $$r_\theta=\left(\begin{array}{cc} \cos\theta&\sin\theta\\-\sin\theta &\cos\theta
       \end{array}\right)$$ denote the rotation subgroup. 

The Teichm\"uller flow acts by expanding the horizontal component of saddle 
connections by a factor of $e^t$ and contracting the vertical components by $e^t$.  
For $\sigma$ a saddle connection we will also use the notation $g_tr_\theta \sigma$ 
for the action on saddle connections.  The action of $SL(2,\R)$ is linear on 
holonomy of saddle connections. 

\begin{definition}
We say a direction $\theta$ is \emph{bounded} if there exists $\epsilon$ 
such that $g_tr_\theta q\in Q^1_\epsilon$ for all $t\geq 0$.  
\end{definition}

\subsection{Conditions for $\beta$-absolute winning}
\begin{definition}
Given a saddle connection $\gamma$ on $q$ we denote by $\theta_\gamma$ 
the angle such that $\gamma$ is vertical with respect to $r_{\theta_\gamma}q$. 
\end{definition}

We can think of the set of saddle connections as a subset of $S^1\times\R$ 
by associating to each $\gamma$ the pair $(\theta_\gamma,|\gamma|)$.
The following proposition gives the equivalence of Theorem~\ref{main} and 
Theorem~\ref{thm:equiv} and will be the motivation for what follows. 

\begin{prop}\label{equiv:cond} 
Let $$S=\{(\theta,L): \theta 
  \text{ is the vertical direction of a saddle connection of length } L\}.$$ 
Then $\underset{(\theta,L)\in S}{\inf}L^2|\theta-\psi|>0$, if and only if 
  $\psi$ determines a bounded direction.
\end{prop}
\begin{proof}
If $(\theta,L)\in S$, let $c=|\theta-\psi|$.  We can assume $c\le\frac{\pi}{4}$.  
Then the length of the saddle connection in $g_tr_{\psi}q$ coming from $(\theta,L)$ 
is $\max\{e^t\sin(c)L,e^{-t}L\cos(c)\}$ and minimized in $t$ when equality of 
the two terms holds; that is when $e^{-t}= \sqrt{\tan(c)}$.  At this time the length is 
$$L\sqrt {\sin(c)\cos(c)}=L\sqrt{\frac{\sin(2c)}{2}}\geq L\sqrt{\frac{c}{2}}.$$  
So if $c>\frac{\delta}{L^2}$, for some $\delta>0$, then 
  $$\max\{e^t\sin(c)L,e^{-t}L\cos(c)\}>\sqrt{\frac{\delta}{2}}.$$
Conversely, if the minimum length $L\sqrt{\frac{\sin(2c)}{2}}$ is bounded below 
by some $\ell_0$ then the difference in angles $c$ satisfies 
  $$2c\geq \sin(2c)\geq\frac{2\ell_0^2}{L^2}.$$
\end{proof}

\subsection{Complexes}
In this section we fix a quadratic differential $q$.  Let $\Gamma$ be a collection of 
saddle connections of $q$, any two of which are disjoint except possibly at a common zero.  
Let $K$ be the simplicial complex having $\Gamma$ as its set of $1$-simplices and 
whose $2$-simplices consist of all triangles that have all three edges in $\Gamma$.  
By a \emph{complex} we mean any simplicial complex that arises in this manner.  
We shall also use the same term to mean the closed subset $K$ of the surface given by 
the union of all simplices; in this case, we call $\Gamma$ a \emph{triangulation} of $K$.  
We shall often leave it to the context to determine which sense of the term is intended.  
For example, ``a saddle connection in $K$" refers to an element of $\Gamma$, 
whereas ``the interior of $K$" refers to the largest open subset contained in $K$, 
which may be empty.  An edge $e$ is in the topological boundary $\partial K$ of $K$ 
if any neighborhood of an interior point of $e$ intersects the complement of $K$.  
We note that the boundary of a complex may fail to satisfy the requirement that 
a triangle with edges in the complex is also included in the complex, as happens 
when $K$ is simply a triangle.  

We distinguish between \emph{internal} saddle connections in $\partial K$, which 
lie on the boundary of a $2$-simplex in $K$ and \emph{external} ones, which do not.  
Note that a triangulation $\Gamma$ of $K$ may contain both internal and external 
saddle connections.  The remaining saddle connections are on the boundary of 
two $2$-simplices, and we refer to them as \emph{interior} saddle connections, 
which, of course, depend on the choice of $\Gamma$.  Each internal saddle 
connection comes with a \emph{transverse orientation}, which is determined by 
the choice of an inward normal vector at any interior point of the segment.  
The interior of $K$ is determined by the data consisting of $\partial K$, 
the subdivision into internal and external saddle connections, together with 
the choice of transverse orientation for each internal saddle connection.  
Simplicial homeomorphisms respect these notions in the obvious sense, 
while simplicial maps generally do not.  

\begin{definition}
The {\em level} of a complex is the number of edges in any triangulation.
\end{definition}

An easy Euler characteristic argument says that the level $M$ is well defined 
and is bounded by $6g-6+3n$, where $n$ is the number of zeroes. 

We say two complexes are \emph{topologically equivalent} if they determine the 
same closed subset of the surface.  Otherwise, they are \emph{topologically distinct}.  
\begin{lemma}\label{lem:distinct}
If $K_1$ and $K_2$ are topologically distinct complexes of the same level, 
then any triangulation of $K_2$ contains a saddle connection $\gamma\in K_2$ 
that intersects the exterior of $K_1$, i.e. $\gamma\not\subset K_1$.  
\end{lemma}
\begin{proof}
Arguing by contradiction, we suppose that the conclusion does not hold.  
Then there is a triangulation of $K_2$ such that every edge is contained in $K_1$.  
It would follow that $K_2\subset K_1$, and properly so, since they are topologically 
distinct.  By repeatedly adding saddle connections $\sigma\subset K_1$ that are disjoint 
from those in $K_2$, we can extend the triangulation of $K_2$ to one of $K_1$ to obtain 
one where the number of edges is strictly greater than the level of $K_1$.  This contradicts 
the fact that any two triangulations of a complex contains the same number of edges.  
\end{proof}

A \emph{path in $\Gamma$} refers to a sequence of edges in $\Gamma$ such that 
the terminal endpoint of the previous edge coincides with the initial endpoint of 
the next edge.  We may also think of it as a map of the unit interval into $X$.  
The \emph{combinatorial length} of a path in $\Gamma$ refers to the number 
of edges in the sequence, including repetitions.  For a homotopy class of paths 
with endpoints fixed at the zeroes of $q$ we define the \emph{combinatorial length} 
to be the minimum combinatorial length of a path in $\Gamma$ in the homotopy class.  
We denote the combinatorial length of a saddle connection by $|\gamma|_\Gamma$. 

To show that combinatorial and flat lengths are comparable, we first need a lemma.
\begin{lemma}\label{lem:iet}
For any saddle connection $\sigma$ there is $\delta=\delta(\sigma)>0$ such that 
the length of a geodesic segment with endpoints in $\sigma$ but otherwise not 
contained in $\sigma$ is at least $\delta$.  
\end{lemma}
\begin{proof}
For any small $\delta>0$ take the $\delta$ neighborhood of $\sigma$.  This is simply connected if $\sigma$ has distinct endpoints and is an annulus if the endpoints coincide. Then any geodesic starting and ending on $\sigma$ must leave the neighborhood; otherwise the geodesic and a segment of $\sigma$ would bound a disc, which is impossible. 
\end{proof}

\begin{definition}
Given $q$, let $L_0=L_0(q)$ denote the systole and 
for $\Gamma$ a triangulation of a complex $K$ let $L_1=L_1(\Gamma,q)$ 
denote the length of the longest edge of $\Gamma$. 
\end{definition}
\begin{lemma}\label{lem:comparable}
Let $\delta=\delta(\Gamma)$ denote the minimum of the constants given by 
Lemma~\ref{lem:iet} associated to each saddle connection in $\Gamma$.  
There are constants $\lambda_2>\lambda_1>0$ depending on $L_0,L_1$ 
and $\delta$ such that for any saddle connection $\gamma\subset\Gamma$, 
  $\lambda_1|\gamma|\le |\gamma|_\Gamma\le \lambda_2|\gamma|$.  
\end{lemma}
\begin{proof}
Since $\gamma$ is the geodesic in its homotopy class, its length is bounded 
  above by $L_1|\gamma|_\Gamma$.  Hence, we may take $\lambda_1=1/L_1$.  
For the other inequality, 
$$|\gamma|_\Gamma\le N+2$$ where $N$ is the number of times $\gamma$ 
  crosses a saddle connection in $\Gamma$.  Since one of these saddle connections 
  is crossed at least $[N/e]$ times, where $e$ is the number of elements in $\Gamma$, 
  we have $$|\gamma|\ge([N/e]-1)\delta.$$  
First, if $N\ge4e$ then $|\gamma| \ge \frac{N}{2e}\delta \ge 2\delta$ so that 
  $$|\gamma|_\Gamma \le \frac{2e}{\delta}|\gamma| + 2 \le \frac{2e+1}{\delta}|\gamma|.$$  
On the other hand, if $N<4e$ then $|\gamma|_\Gamma<4e+2$ whereas $|\gamma|$ 
  is bounded below by the  $L_0$.  Hence, the lemma holds with 
  $$\lambda_2=\max\left(\frac{2e+1}{\delta},\frac{4e+2}{L_0}\right).$$
\end{proof}

\begin{lemma}\label{lem:small}
There exists $\epsilon_0$ such that a $3\epsilon_0$-complex must have strictly 
fewer than $6g-6+3n$ saddle connections.  
\end{lemma}
\begin{proof}
We can triangulate the surface by disjoint saddle connections.  If the surface can 
be triangulated by edges of length $\epsilon_0$ then there is a bound in terms of 
$\epsilon_0$ for the area.  However we are assuming that the area of $q$ is one.
\end{proof}

\begin{lemma}
Given $q$, there is a number $M<6g-6+3n$ such that for any $\epsilon\leq\epsilon_0$, 
$M$ is the maximum level of any $\epsilon$-complex for any $g_tr_\theta q$.  
\end{lemma}

The following will be applied to complexes on the surface $g_tr_\theta q$ for 
  some suitable choice of $t$ and $\theta$. 
\begin{lemma}\label{lem:sigma}
Let $K$ be an $\eps$-complex and $\gamma\not\subset K$, i.e. a saddle connection 
that intersects the exterior of $K$.  Then there exists a complex $K'=K\cup\{\sigma\}$ 
formed by adding a disjoint saddle connection $\sigma$ satisfying 
$h(\sigma) \le h(\gamma) + 3\eps$ and $v(\sigma) \le v(\gamma) + 3\eps$.  
\end{lemma}
\begin{proof}
We have that $\gamma$ must be either disjoint from $K$ or cross the boundary of $K$.

{\bf Case I}.
 $\gamma$ is disjoint from $K$.
Add $\gamma$ to $K$ to form $\tilde K$. 
It is clear that the estimate on lengths  holds.

{\bf Case II}
 $\gamma$ intersects $\partial K$  crossing  $\beta \subset \partial K$ at a point $p$ 
dividing $\beta$ into segments $\beta_1,\beta_2$.  

{\bf Case IIa}  One endpoint $p'$ of $\gamma$ lies in the exterior of $K$.  
Let $\hat \gamma$ be the segment of $\gamma$  that goes from  $p'$ to $p$.  We consider the homotopy class of paths $\hat\gamma*\beta_i$ which is the segment $\hat\gamma$ followed by $\beta_i$.  Together with $\beta$ they bound a simply connected domain $\Delta$.  Replace each  path by the geodesic $\omega_i$ joining the endpoints in the homotopy class.  
 Then $\partial \Delta$ is made up of at most $M$ saddle connections $\sigma$ all of which have their horizontal and vertical lengths bounded by the sum of the horizontal and vertical lengths of $\gamma$ and $\beta$.  
If some $\sigma\notin K$ we add it to form $K'$.  
It is clear that the estimate on lengths  holds. 

The other possibility is that $\Delta\subset \partial K$. 
It cannot be the case that $\Delta$ is a triangle, since then $\Delta$ would be 
a subset of $K$, contradicting the assumption on $\gamma$.   
Since the edges of $\Delta$ all have length at most $\eps$ we can find 
a diagonal $\sigma$ in $\Delta$ of length at most $2\eps$ and add it to form $K'$.
See Figure~1.  
\begin{center}
\begin{figure}\label{fig:IIa}
\includegraphics{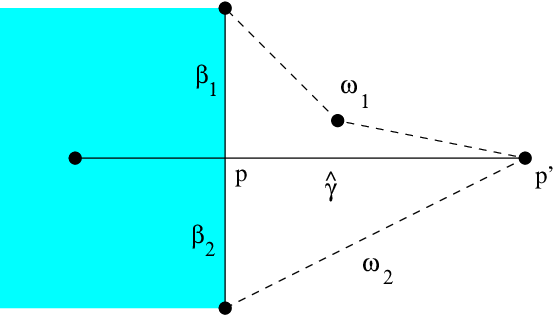}
\caption{Case IIa.}
\end{figure}
\end{center}

{\bf Case IIb} 
Both endpoints of $\gamma$ lie in $K$. 
Let $\gamma$ successively cross $\partial K$ at $p_1,p_2$, and let $\gamma_1$ be 
the segment of $\gamma$ lying in  the exterior of $K$ between $p_1$ and $p_2$.  

The first  case  is where $p_1,p_2$ lie on different $\beta_1,\beta_2$ which have  
endpoints $q_1^1,q_1^2$ and $q_2^1,q_2^2$.  Then $p_1,p_2$  divide $\beta_i$ 
into segments $\beta_i',\beta_i''$.  We can form a homotopy class $\beta_1'*\gamma_1*\beta_2'$ joining $q_1^1$ to $q_2^2$ and a homotopy class $\beta_1''*\gamma_1*\beta_2''$ joining $q_1^2$ to $q_2^1$. cWe replace these with their geodesics with 
the same endpoints and then together with $\beta_1,\beta_2$ they bound a simply 
connected domain.  We are then in a situation similar to Case IIa. 
See Figure~2.  
\begin{center}
\begin{figure}\label{fig:IIbi}
\includegraphics{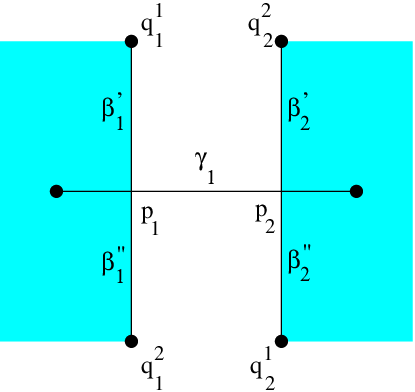}
\caption{First subcase of Case IIb.}
\end{figure}
\end{center}

The last case is that $p_1,p_2$ lie on the same  saddle connections $\beta$ 
of $\partial K$.  Let $\hat\beta$ be the segment between $p_2$ and $p_1$.  
Let $\beta_1$ and $\beta_2$ be the segments joining the endpoints $q_1,q_2$ 
of $\beta$ to $p_1,p_2$.  Find the geodesic in the homotopy class of 
$\beta_1*\gamma_1*\beta_2$ joining $q_1$ to $q_2$ and the geodesic 
in the class of the loop $\beta_1*\gamma_1*\beta_2*\beta^{-1}$ from $q_1$ to itself.  
These two geodesics together with $\beta$  bound a simply connected domain.  
The analysis is similar to the previous cases.  
See Figure~3.  
\begin{center}
\begin{figure}\label{fig:IIbii}
\includegraphics{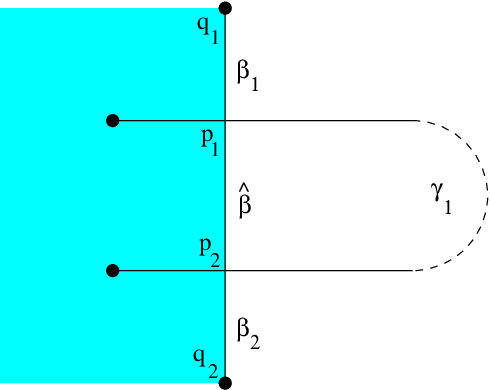}
\caption{Second subcase of Case IIb.}
\end{figure}
\end{center}
\end{proof}

Now fix the base surface $q$.  
All angles and lengths will be measured on the base surface.  
We shall often let $K$ denote a complex equipped with a triangulation 
without explicit mention the choice of triangulation, as in the 
statement and proof of Lemma~\ref{lem:sigma}.  
\begin{definition}
Denote by $L(K)$ the length of the longest saddle connection in $K$.
Let $\theta(K)$ the angle that makes the longest saddle connection vertical. 
\end{definition}

We assume that  the complexes considered now have the property that for any saddle 
connection $\gamma \in K$ we have $|\theta_{\gamma}-\theta(K)|\leq \frac{\pi}{4}$.  
This implies that measured with respect to the angle $\theta(K)$ we have 
$|\gamma|=v(\gamma)$.  In other words the vertical component is larger than 
the horizontal component.  This will exclude at most finitely many complexes from 
our game and these will be excluded in any case by our choice of $c_{M+1}$ in 
Theorem~\ref{blocking strategy}.  

\begin{definition}  We say a complex $K$ is \emph{$\epsilon$-shrinkable} if for all saddle connections $\beta$ of $K$ if we let $h(\beta)$ be the component of the holonomy vector 
in the direction perpendicular to the direction $\theta(K)$ of the longest saddle connection, 
then $h(\beta)\leq \frac{\epsilon^2}{L(K)}$. 
\end{definition}

We note that this condition could equally well be stated as follows. 
For any  saddle connection $\beta$ of $K$ we have 
  $|\theta_\beta-\theta(K)|\leq \frac{\epsilon^2}{|\beta|L(K)}$.

The following is immediate. 
\begin{lemma}
If $\epsilon_1<\epsilon_2$ and  $K$ is $\epsilon_1$-shrinkable, 
then it is $\epsilon_2$-shrinkable.  
\end{lemma}

\begin{definition}
A complex $K$ and a saddle connection $\gamma$ that intersects the exterior of $K$ 
are \emph{ jointly $\epsilon$-shrinkable} if $K$ is $\epsilon$-shrinkable and 
\begin{itemize}
\item if $|\gamma|\leq L(K)$ then $|\theta(K)-\theta_\gamma|\leq \frac{\epsilon^2}{|\gamma|L(K)}$.
\item if $L(K)<|\gamma|$ then $|\theta_\gamma-\theta_\omega|\leq \frac{\epsilon^2}{|\gamma||\omega|}$ for all $\omega\in K$.
\end{itemize}  
\end{definition}

The next lemma says that if the longest saddle connections of each of two complexes 
have comparable lengths and the angles between these saddle connections is not too 
large, then the complexes can be combined to form another shrinkable complex.  
\begin{lemma}\label{lem:combine}
Let $K_1$ and $K_2$ be $\eps$-shrinkable complexes of level $i$ satisfying 
\begin{equation*}
  \left|\theta(K_1)-\theta(K_2)\right| < \frac{\rho_1\eps^2}{L(K_1)L(K_2)} 
  \quad\text{ and }\quad L(K_1) \le L(K_2) < \rho_2 L(K_1)
\end{equation*}
for some $\rho_1>3$ and $\rho_2>3$ and assume they are topologically distinct.  
Then there 
is an $\eps'$-shrinkable complex $K'$ of one level higher satisfying 
  $L(K') < \rho'_2 L(K_1)$ where 
\begin{equation}\label{def:K'}
  \eps'=(16\rho_1\rho'_2)^{1/2}\eps  \quad\text{ and }\quad 
  \rho'_2 = \sqrt{4\rho_2^2+9\rho_1^2\eps^4/L(K_1)^4}.  
\end{equation}
\end{lemma}
\begin{proof}
By Lemma~\ref{lem:distinct} there is a saddle connection $\gamma\in K_2$ such that
   $\gamma\not\subset K_1$.  Let  $\theta=\theta(K_1)$ and $t=\log(L(K_1)/\eps)$.  Then 
\begin{equation*}
  \left| \theta(K_1) - \theta_\gamma\right| 
      < \frac{\rho_1\eps^2}{L(K_1)L(K_2)} + \frac{\eps^2}{|\gamma| L(K_2)} 
      < \frac{2\rho_1\eps^2}{|\gamma| L(K_1)}
\end{equation*}
so that
\begin{equation*}
  h_\theta(\gamma)\cdot\frac{L(K_1)}{\eps} < 2\rho_1\eps
  \quad\text{ and }\quad 
  v_\theta(\gamma)\cdot\frac{\eps}{L(K_1)} < \rho_2\eps.  
\end{equation*}
We apply Lemma~\ref{lem:sigma} to produce a new saddle connection $\sigma$. On $g_tr_\theta X$ we have $\sigma_{\theta,t}=g_tr_\theta\sigma$ satisfies 
\begin{equation*}
  h(\sigma_{\theta,t}) < (2\rho_1+3)\eps < 3\rho_1\eps
  \quad\text{ and }\quad 
  v(\sigma_{\theta,t}) < (\rho_2+3)\eps < 2\rho_2\eps
\end{equation*}
  which implies that 
\begin{align*}
  |\sigma| &= \sqrt{\left(h(\sigma_{\theta,t})\cdot\frac{\eps}{L(K_1)}\right)^2 
                                + \left(v(\sigma_{\theta,t})\cdot\frac{L(K_1)}{\eps}\right)^2} \\
                  &< \sqrt{(3\rho_1\eps^2/L(K_1))^2+4\rho_2^2L(K_1)^2} = \rho'_2L(K_1) 
\end{align*}
  so that $K'=K\cup\{\sigma\}$ satisfies 
\begin{equation}\label{ieq:K'}
  L(K') = \max( L(K_1),|\sigma| ) < \rho'_2L(K_1).  
\end{equation}
If $|\sigma|<L(K_1)$ we have $K'$ is $\eps'$-shrinkable because 
  $$ \left| \theta(K_1)-\theta_\sigma\right| \leq  2 \frac{ h_{\theta}(\sigma)}{|\sigma|}\leq  \frac{6\rho_1\eps^2}{|\sigma| L(K_1)}$$ 
while if $|\sigma|\ge L(K_1)$ it follows that $K'$ is $\eps'$-shrinkable because $\theta(K')=\theta_\sigma$ 
  and for every $\xi\in K_1$ 
\begin{align*}
  \left| \theta(K')-\theta_\xi \right| 
      &\le \left| \theta_\sigma - \theta(K_1) \right| + \left| \theta(K_1) - \theta_\xi \right| \\
      &< \frac{6\rho_1\eps^2}{|\sigma| L(K_1)} + \frac{\eps^2}{|\xi| L(K_1)} 
      < \frac{6\rho_1\eps^2}{|\xi| L(K_1)} 
      < \frac{6\rho_1\rho'_2\eps^2}{|\xi| L(K')}
\end{align*}
  where $|\sigma|\ge L(K_1)\ge|\xi|$ and (\ref{ieq:K'}) were used in last two inequalities.  
\end{proof}

The following symmetric version allows us to bypass Lemma~\ref{lem:distinct}.  
\begin{lemma}\label{lem:combine2}
Let $K_1$ and $K_2$ be $\eps$-shrinkable complexes of level $i$ satisfying 
\begin{equation*}
  \left|\theta(K_1)-\theta(K_2)\right| < \frac{\rho_1\eps^2}{L(K_1)L(K_2)} 
  \quad\text{ and }\quad \rho_2^{-1}L(K_1) \le L(K_2) < \rho_2 L(K_1)
\end{equation*}
for some $\rho_1>3$ and $\rho_2>3$ and assume they are topologically distinct.  
Then there is an $\eps'$-shrinkable complex $K'$ of one level higher satisfying 
  $L(K') < \rho'_2 L(K_1)$ where 
\begin{equation}\label{def:K'2}
  \eps'=(8\rho_*\rho'_2)^{1/2}\eps, \quad 
  \rho'_2 = \sqrt{4\rho_2^2+9\rho_*^2\eps^4/L(K_1)^4} \quad\text{ and }\quad 
  \rho_*=\max(\rho_1,\rho_2).  
\end{equation}
\end{lemma}
\begin{proof}
The only place where $L(K_1)\le L(K_2)$ was used in the previous proof was 
  at the last inequality in the first displayed line.  The entire proof goes through 
  if every occurrence of $\rho_1$ is replaced with $\rho_*$.  
\end{proof}

In what follows we will be considering shrinkable complexes.  In each combinatorial equivalence class of such shrinkable complexes we will consider the complex $K$ which minimizes $L(\cdot)$ and the corresponding angle $\theta(\cdot)$.  
We note that this complex is perhaps not unique and so there is ambiguity in $\theta(K)$
but this will not matter. 

We let $L(\partial K)$ denote the length of the longest saddle connection on the boundary of $K$. 
We also let $\theta(\partial K)$ the angle that makes  the longest vertical.

\section{Proofs of Theorem~\ref{main} and Theorem~\ref{blocking strategy}}
We are now ready to begin the proof of Theorem~\ref{main}.  
It is based on Theorem~\ref{blocking strategy} whose  statement and proof were suggested in the outline.  

\begin{thm}\label{blocking strategy}  For  all $\beta$ sufficiently small  and given  Bob's first move $I_1$ in the game, there exist  positive constants $c_i$, $i=1,\ldots, M+1$, and  a strategy for Alice  
such that regardless of the choices $I_j$ made by Bob, the following will hold.   For all   
$\beta c_i$-shrinkable level $i$-complexes $K$ if 
$$L(K)L(\partial K)|I_j|< c_i^2,$$ then 
$$d(\theta(K),I_j)>\frac{\beta c_i^2}{4L(K)L(\partial K)}.$$ 
\end{thm}

\begin{proof} 
We assume $\beta<1/12$.  (The value $1/12$ is chosen so that some inequalities that appear in the proof are satisfied).   Let $L_0$ denote the length of the shortest saddle connection on $(X,q)$.  
Let $0<c_1<\dots<c_{M+1}$ be given by $$c_i=\beta^{N_i}c_{i+1}$$ and 
    $$c_{M+1} = \min( L_0\beta^{N_M},L_0|I_1|^{1/2}\eps_0 ),$$ 
  where $N_i$ are defined recursively by $N_1=6$ and 
  $N_{i+1} = 6 + 2(N_1+\dots+N_i)$.  (Again the value $6$ is chosen only so that a particular inequality is satisfied) . 
These $N_i$ are chosen so that 
\begin{equation}\label{def:c_i}
  \frac{c_{i+1}}{c_i} = \beta^{-6}\left(\frac{c_i}{c_1}\right)^2.  
\end{equation}
We remark that this last equation will be used in Step $2$ of the proof. All that is needed is an inequality, but to simplify matters we present it as an equality. 

Let $\cE_i$ be the set of all marked $\beta c_i$-shrinkable complexes of level $i$.  
Given $I_j$, we let 
\begin{equation*}
  \cA_i(j) := \left\{K\in\cE_i : d(\theta(K),I_j) > \frac{\beta c_i^2}{4L(K)L(\partial K)} \right\} 
\end{equation*}
and 
\begin{equation*}
  \Omega_i(j) := \left\{ K\in\cE_i\setminus\cA_i(j) : 
     \frac{c_i^2}{L(K)L(\partial K)} \le |I_j|< \frac{\beta^{-1}c_i^2}{L(K)L(\partial K)} \right\} 
\end{equation*}
and
\begin{equation*}
  z_i(j) := \frac{\inf \Theta_i(j) + \sup \Theta_i(j)}{2} 
     \quad\text{ where }\quad \Theta_i(j)=\{\theta(K):K\in\Omega_i(j) \}.  
\end{equation*}
Alice chooses $M$ intervals of length $\beta|I_j|$ centered at the points $z_i(j),\; i=1,\dots,M$.  

The restatement of theorem then is that for every $j\ge1$:  
\begin{align*}
\tag{$P_j$} \forall i\in\{1,\dots,M\} \; \forall K\in\cE_i \qquad
                                 L(K)L(\partial K)|I_j| < c_i^2 \;\;\Longrightarrow\;\; K\in \cA_i(j)
\end{align*}
Note that ($P_j$) holds for all $j<j_0:=\min\{k: |I_k|<\beta^{2N_M}\}$ because 
  $$L(K)L(\partial K)|I_j| \ge L_0^2\beta^{2N_M} \ge c_{M+1}^2 > c_i^2$$ 
  while if $j_0=1$, we note that $L(K)L(\partial K)|I_1|\ge L_0^2|I_1| \ge c_{M+1}^2>c_i^2.$  
  
We proceed by induction and suppose that $j\ge j_0$ and that ($P_{j-1}$) holds.  

\textit{Step 1.} For any $K\in\Omega_i(j)$, 
\begin{equation}\label{eq:boundary}
  \frac{L(K)}{L(\partial K)} < \beta^{-2}\left(\frac{c_i}{c_1}\right)^2.  
\end{equation}
Consider first  the case 
\begin{equation}
\label{eq:case1}
L(\partial K)^2|I_j| < \beta c_1^2.
\end{equation}  
Then $L(\partial K)^2|I_{j-1}| < c_1^2$ so that ($P_{j-1}$) implies the longest saddle 
connection on $\partial K$ belongs to $A_1(j-1)$, meaning 
\begin{equation}
\label{eq:distancebigger}  d(\theta(\partial K),I_{j-1}) > \frac{\beta c_1^2}{4L(\partial K)^2}.  
\end{equation}
But since $K$ is $\beta c_i$-shrinkable and not in $\cA_i(j)$, the triangle inequality implies 
\begin{align*}
  d(\theta(\partial K),I_{j-1}) &\le \left| \theta(\partial K) - \theta(K) \right| + d(\theta(K),I_j) \\
             &\le \frac{\beta^2c_i^2}{L(K)L(\partial K)} + \frac{\beta c_i^2}{4L(K)L(\partial K)} 
             \le \frac{\beta c_i^2}{L(K)L(\partial K)}
\end{align*}
which together with (\ref{eq:distancebigger}) and the fact that $\beta<\frac{1}{3}$ implies (\ref{eq:boundary}).  Suppose then that (\ref{eq:case1}) does not hold.  Then we have 
\begin{equation*}
  \frac{L(K)}{L(\partial K)} = \frac{L(K)L(\partial K)|I_j|}{L(\partial K)^2|I_j|} 
                                             < \frac{\beta^{-1}c_i^2}{\beta c_1^2}
\end{equation*}
so that (\ref{eq:boundary}) holds in this case as well.

\textit{Step 2.} Any pair $K_1,K_2\in\Omega_i(j)$ are topologically equivalent.  \\
For any $K_1,K_2\in\Omega_i(j)$, we have 
\begin{equation}\label{ieq:ratio}
  \frac{L(K_2)}{L(\partial K_1)} < \beta^{-1}\frac{L(K_1)}{L(\partial K_2)}.  
\end{equation}
Multiplying (\ref{ieq:ratio}) by $L(K_2)/L(\partial K_1)$ and invoking (\ref{eq:boundary}), 
  we get 
\begin{equation*}
  \frac{L(K_2)}{L(\partial K_1)} < \beta^{-5/2}\left(\frac{c_i}{c_1}\right)^2 
\end{equation*}
  so that, exploiting $2<\beta^{-1/2}$, and the above inequality we have 
\begin{equation*}
  \left| \theta(K_1) - \theta(K_2) \right| \le 2|I_j| < \frac{2\beta^{-1}c_i^2}{L(K_1)L(\partial K_1)} 
     < \frac{\beta^{-4}(c_i/c_1)^2c_i^2}{L(K_1)L(K_2)}.  
\end{equation*}
On the other hand again multiplying (\ref{ieq:ratio}) by $L(K_2)/L(K_1)$ and applying (\ref{eq:boundary}), we obtain 
\begin{equation*}
  \left(\frac{L(K_2)}{L(K_1)}\right)^2 < \beta^{-1}\frac{L(K_2)L(\partial K_1)}{L(\partial K_2)L(K_1)} 
       < \beta^{-3}\left(\frac{c_i}{c_1}\right)^2.  
\end{equation*}

Now suppose $K_1$ and $K_2$ are topologically distinct.  We shall derive a contradiction.  
Without loss of generality, we may assume $L(K_1)\le L(K_2)$.  
The hypotheses of Lemma~\ref{lem:combine} are then satisfied with the parameters 
  $$\eps=c_i, \qquad \rho_1=\beta^{-4}\left(\frac{c_i}{c_1}\right)^2, 
               \quad\text{ and }\quad \rho_2=\beta^{-2}\left(\frac{c_i}{c_1}\right).$$
Consider the two terms under the radical in the expression for $\rho'_2$ in (\ref{def:K'}).  
Note that the second being dominated by the first is equivalent to 
  $L(K_1)^2>\displaystyle \frac{3\rho_1}{2\rho_2}c_i^2$.  
we have  $$|I_j| \le |I_{j_0}| < \beta^{2N_M} = \left(\frac{c_M}{c_{M+1}}\right)^2 
                                             < \left(\frac{c_1}{c_i}\right)^4 < \frac{2\rho_2}{3\rho_1},$$ 
  the first inequality a conseqeunce of (\ref{def:c_i}).  Since  $L(K_1)^2\ge L(K_1)L(\partial K_1)\ge c_i^2/|I_j|$, it now follows that 
  $$\rho'_2 < 3 \rho_2.$$  
Also, since $\rho_1=\rho_2^2$ and $\beta^2 c_{i+1}=\rho_2^2c_i$, we have 
\begin{equation}\label{ieq:eps'}
  \eps' = (8\rho_1\rho'_2)^{1/2}c_i < 2\sqrt6\rho_2^{3/2}c_i 
                                      < \frac{\beta c_{i+1}}{\sqrt{6\rho_2}} < \beta c_{i+1}
\end{equation}
  where we used $\beta<1/12$ in the middle inequality.  

By Lemma~\ref{lem:combine}, there exists a $\beta c_{i+1}$-shrinkable complex $K'$ 
  of level $i+1$ satisfying $$L(K')<3\rho_2L(K_1).$$  
Since there are no $c_{M+1}$-shrinkable complexes of level $M+1$, we have our 
  desired contradiction when $i=M$.  For $i<M$, we have 
\begin{align*}
  L(K')L(\partial K')|I_{j-1}| < \frac{9\rho_2^2L(K_1)^2|I_j|}{\beta}
                                          < \frac{9\rho_2^2c_i^2L(K_1)}{\beta^2L(\partial K_1)}
                                          < 9\beta^{-8}c_i^2\left(\frac{c_i}{c_1}\right)^4 < c_{i+1}^2,
\end{align*}
 where in the last inequality we have used (\ref{def:c_i}) and (\ref{eq:boundary}).  The induction hypothesis ($P_{j-1}$) implies $K'\in\cA_{i+1}(j-1)$, meaning 
  $$d(\theta(K'),I_{j-1}) > \frac{\beta c_{i+1}^2}{4L(K')L(\partial K')} 
                                     \ge \frac{\beta c_{i+1}^2}{4L(K')^2}.$$  
On the other hand by  (\ref{ieq:eps'}), $K'$ is in fact $\displaystyle \frac{\beta c_{i+1}}{\sqrt{6\rho_2}}$-shrinkable, 
  and since $c_{i+1}>\sqrt{6\rho_2}c_i$, we have 
\begin{align*}
  d(\theta(K',I_{j-1}) &\le \left| \theta(K') - \theta(K_1) \right| + d(\theta(K_1),I_j) \\
      &< \frac{\beta^2c_{i+1}^2}{6\rho_2L(K')L(K_1)} + \frac{\beta c_i^2}{L(K_1)L(\partial K_1)}
       < \frac{\beta c_{i+1}^2}{12\rho_2L(K')L(K_1)} < \frac{\beta c_{i+1}^2}{4L(K')^2}
\end{align*}
which contradicts the previously displayed inequality.  This finishes the proof of Step 2.  

\textit{Step 3.} For each $i=1,\dots,M$, we have $\diam \Theta_i(j) < \frac\beta2|I_j|$. \\ 
Assume $\Omega_i(j)\neq\emptyset$ and fix a complex $K_0$ in it.  
The previous step implies for any $K\in\Omega_i(j)$, we have $\partial K=\partial K_0$.  
Moreover, the longest saddle connection on $\partial K_0$ belongs to $K$ so that 
  since $K$ is $\beta c_i$-shrinkable, we have (using $\beta<1/5$)
  $$\left| \theta(K) - \theta(\partial K_0) \right| < \frac{\beta^2c_i^2}{L(K)L(\partial K_0)} 
                                                 = \frac{\beta^2 c_i^2}{L(K)L(\partial K)} < \frac\beta5 |I_j|.$$
Thus, $\Theta_i(j)$ is inside a ball of radius $\frac{\beta}{5}|I_j|$, so that 
  its diameter is $\le\frac{2\beta}{5}|I_j|<\frac\beta2|I_j|$.  

\textit{Step 4.}  We now show that ($P_j$) holds. \\
Suppose $K\in\cE_i$ is such that $L(K)L(\partial K)|I_j| < c_i^2$.  
Since $|I_j| \ge \beta|I_{j-1}|$, we have  $L(K)L(\partial K)|I_{j-1}| < \beta^{-1}c_i^2$.  
There are two cases.  If $L(K)L(\partial K)|I_{j-1}| < c_i^2$, then ($P_{j-1}$) implies 
  $K\in\cA_i(j-1)\subset\cA_i(j)$ and we are done.  
Otherwise, we conclude that $K\in\Omega_i(j-1)$ so that, by the previous step, 
  $\theta(K)$ lies in an interval of length $<\frac\beta2|I_{j-1}|$ centered about $z_i(j-1)$.  
Since Bob's interval $I_j$ must be disjoint from the interval of length $\beta|I_{j-1}|$ 
  centered at $z_i(j-1)$ chosen by Alice, we have 
  $$ d(\theta(K),I_j) > \frac\beta4|I_{j-1}| > \frac{\beta c_i^2}{4L(K)L(\partial K)}.$$  
In any case, we have $K\in\cA_i(j)$.  
\end{proof}

\begin{proof}[Proof of Theorem \ref{main}]
By Theorem \ref{blocking strategy} we are able to ensure that  for any level $i$ complex $K_i$, we have $$ \max\{L(\partial K) \cdot L(K) \cdot |I_{j}|, L(\partial K) \cdot L(K) \cdot d(\theta(K),I_{j})\}> \frac{\beta c_i^2}{4}.$$  In particular this holds when $i=1$.  Since there is only one saddle connection in a 1-complex,  and since  for any fixed  saddle connection $\gamma$, $|\gamma|^2|I_{j}|\to 0$ as $j\to\infty$,  we conclude that  for all but finitely many intervals $I_{j}$  we have
$$|\gamma|^2d(\theta_{\gamma},I_{j})=\max\{|\gamma|^2|I_{j}|, |\gamma|^2 d(\theta_{\gamma}, I_{j})\}>\frac{\beta c_1^2}{4}.$$ 
Thus if $\phi=\cap_{l=-1}^\infty I_{l}$ is the point we are left with at the end of the game, and $\gamma$ is a saddle connection, then $|\gamma|^2|\theta_{\gamma}-\phi|>\frac{\beta c_1^2}{4}$, which by Proposition \ref{equiv cond} establishes Theorem \ref{main}.
\end{proof}

\section{Playing the Game in the Stratum}

In this section we prove a theorem that as a corollary will imply  Theorem~\ref{thm:fullspace}, Theorem~\ref{thm:riemann},  Theorem~\ref{thm:iet} and Theorem~\ref{thm:pmf}.
In the general situation we will be  playing the game in a 
subset of a stratum $Q^1(k_1,\ldots, k_n,\pm)$. 
In the case of  Theorem~\ref{thm:fullspace} it will be the entire stratum.  
In the case of Theorem~\ref{thm:riemann} and Theorem~\ref{thm:pmf} it is the entire space of quadratic differentials on a fixed Riemann surface, and in the case of Theorem~\ref{thm:iet} a subset  of  the space of Abelian differentials on a compact Riemann surface.  What these examples have in common is that there is a $S^1$ action on the space given by $q\to e^{i\theta}q$. This will allow us to use the ideas of the previous section.

\subsection{Product structure and metric}
Given a quadratic differential $q_0$ that belongs to a stratum $Q^1(k_1,\ldots, k_n,\pm)$  and a triangulation $\Gamma=\{e_i\}_{i=1}^{6g-6+3n}$ of it, we have a chart 
  $\vhi:U\to\C^{6g-6+3n}$ on a neighborhood $U$ of $q_0$ in the stratum where 
  the triangulation remains defined, i.e. none of the triangles are degenerate.  

For sufficiently small $U$, 
using holonomy coordinates, we obtain an embedding $$\vhi_\Gamma:U\to\C^{6g-6+3n}$$ whose 
  image is a convex subset of a linear subspace.  
Equip $U$ with the metric induced by the norm $\|\bz\|_\Gamma=\max(|z_1|,\dots,|z_{6g-6+3n}|)$. 
Note that the notation $\|q_1-q_2\|_\Gamma$ for the distance between $q_1$ and $q_2$ in these holonomy coordinates should not be confused with the possibility that $q_1,q_2$ are quadratic differentials on the same Riemann surface in which case $q_1-q_2$ will  refer to vector space subtraction and $\|q_1-q_2\|$ the area.

We note that $U$ and the induced metric depend only on the $6g-6+3n$ homotopy classes relative to the zeroes 
  of the saddle connections in the triangulation.  However a change in homotopy classes will induce a bi-Lipschitz map of metrics and since winning is invariant under bi-Lipschitz maps, we are free to choose any triangulation.

Note that multiplication by $e^{i\theta}$ defines an $S^1$-action that is equivariant 
  with respect to $(U,\vhi)$.  
Let $\pi_a:\vhi(U)\to S^1$ be the map that gives the argument of $e_1$.  
Let $Z=\pi_a^{-1}(\theta_0)$ where $\theta_0=\pi_a(q_0)$ and 
  let $\pi_Z:\vhi(U)\to Z$ be the map that sends $q$ to the unique point of $Z$ 
  that is contained in the $S^1$-orbit of $q$.  
Then $$ \vhi(U) \simeq Z\times S^1 $$ with projections given by $\pi_Z$ and $\pi_a$.  

The metric on $Z$ is the ambient  metric: 
   $$d_Z(q_1,q_2) = \|q_1-q_2\|_\Gamma \quad\text{ for }\quad q_1,q_2\in Z$$ 
The metric on $U$ is given by 
  $$d_U(q_1,q_2) = \max( d_Z(\pi_Z(q_1),\pi_Z(q_2)), d_a(\pi_a(q_1),\pi_a(q_2)) )$$ 
  where $d_a(\cdot,\cdot)$ is the distance on $S^1$ measuring difference in angles.  
This metric has the property that a ball  in the metric $d_U$ is a ball in each factor. 
\begin{definition}
By an $\eps$-\emph{perturbation} of $q$ we mean any flat surface in $U$ whose 
  distance from $q$ is at most $\eps$.  
\end{definition}

We now show that the holonomy of any  {\em any} saddle connection of $q$ is not changed much by an $\eps$-perturbation.
Recall the constant $\lambda_2$, given in Lemma~\ref{lem:comparable} that depends 
  on the  and the choice of triangulation. 
  
\begin{lemma}\label{lem:angle}
Let $q'$ be an $\eps$-perturbation of $q$ and suppose that the homotopy class 
  specified by a saddle connection $\gamma$ in $q$ is represented on $q'$ by a union 
  of saddle connections $\cup_{1}^k\gamma'_i$.  
Then the total holonomy vector  $hol(\cup \gamma'_i)$ makes an angle at most $2\lambda_2\eps$ 
  with the direction of $\gamma$ and and its length differs from that of $\gamma$ 
  by a factor between $1\pm \lambda_2\eps$.  
Also, the direction of the individual $\gamma'_i$ also lie within $2\lambda_2\epsilon$ of $\gamma$. 
\end{lemma}
\begin{proof}
Represent $\gamma$ as a path in the triangulation $\Gamma$ on $q$.  
After perturbation, the total holonomy vectors   $hol(\gamma),hol(\gamma')$ satisfy  $$hol(\gamma')-hol(\gamma)\leq  
  |\gamma|_\Gamma\eps\le \lambda_2|\gamma|\eps,$$ by Lemma~\ref{lem:comparable}.  
Hence the difference in angle is at most
$$\arcsin\left(\frac{\lambda_2\eps|\gamma|}{|\gamma|}\right)\leq 2\lambda_2\eps,$$ proving 
  the first statement.  

For the individual $\gamma'_i$, fix a linear parametrization $q_t, 0\le t\le 1$, 
  so that $q_0=q$ and $q_1=q'$.  Then there are times $0=t_0<t_1<\dots<t_n=1$ 
  and saddle connections $\gamma_j$ on $q_{t_j}$ (that are parallel to other saddle connections on $q_{t_j}$) such that $\gamma_0=\gamma$, 
  $\gamma_n=\gamma'_i$, and, by the first part of the lemma, 
  the angle between $\gamma_j$ and $\gamma_{j+1}$ is at most $$2\lambda_2\eps(t_{j+1}-t_j).$$  
The triangle inequality now implies the angle between the holonomies of $\gamma$ 
  and $\gamma'_i$ is at most $2\lambda_2\eps$.  
\end{proof}

\subsection{Moving complexes}

In the proof of the theorems we will need to move triangulations from one quadratic differential to another in order to play the games.  In such a move, vertices of the triangulation may hit other edges  forcing degenerations.  The following theorem is the mechanism for keeping track of complexes as they move. 
We first note that about each point in the stratum there is a neighborhood where the homotopy class of a saddle connection can be consistently defined.

\begin{thm}
\label{thm:move}
Suppose $q_t; 0\leq t\leq 1$ is a smooth  path of quadratic differentials in a given stratum.  
Suppose $K$ is a complex on $q_0$ (with triangulation $\Gamma$). 
Then there is a complex  $K'$  on $q_1$ with triangulation, denoted $\Gamma'$ and a piecewise linear map $F:K\to K'$ such that 
\begin{enumerate}
\item the homotopy class of every saddle connection of $K$ is mapped by $F$ to a union of saddle connections  on $K'$.
  These saddle connections have the same homotopy class.
\item the closed subset $K'$ depends only on $K$ and the path of quadratic differentials; in particular it does not depend on the choice of triangulation of $\Gamma$ of $K$. 
\end{enumerate}
\end{thm}
\begin{proof}
Let $T$ be the set of $t\ge0$ such that the geodesic representative on $q_t$ of the homotopy class of each saddle connection in $\Gamma$ is realized by a single saddle connection in $q_t$.  Let $A(0)$ be the connected component of $T$ containing $0$.  
For each $t\in A(0)$, let $\Gamma_t$  be the collection of saddle collections in $q_t$ representing these homotopy classes. 
It is easy to see that $\Gamma_t$ is a pairwise disjoint collection and that 
  three saddle connections in $\Gamma$ bound a triangle if and only if the 
  corresponding saddle connections in $\Gamma_t$ bound a triangle.  
Let $K_t$ be the complex determined by $\Gamma_t$.  
The obvious piecewise linear map $f_t:K\to K_t$ is a homeomorphism onto its image.  

Let $t_1=\sup A(0)$.  We claim that the closed set $K_t$ is independent of the choice of triangulation $\Gamma$ for $0<t<t_1$.  Indeed, suppose $\tilde\Gamma$ is another triangulation of $K$ such that for $0<t<t_1$ the geodesic representative on $q_t$ of the homotopy class of each saddle connection in $\tilde\Gamma$ is realized by a single saddle connection on $q_t$.  Let $\tilde f_t:K\to\tilde K_t$ be the simplicial homeomorphism between $K$ and the complex $K_t$ determined by the corresponding collection $\tilde\Gamma_t$ of saddle connections on $q_t$.  
Then $\tilde f_t$ and $f_t$ agree on $\partial K$, so that 
  $\partial K_t = f_t(\partial K) = \tilde f_t(\partial K) = \partial\tilde K_t$. 
Note that $f_t$ and $\tilde f_t$ induce the same transverse orientation 
  on any saddle connection in $\partial K_t$.  
Since $\tilde f_t\circ f_t^{-1}$ restricts to the identity on $\partial K_t$, 
  it maps each connected component of the interior of $K_t$ to itself.  
Hence, the interiors of $K_t$ and $\tilde K_t$ coincide, and therefore 
  $K_t=\tilde K_t$, proving the claim.  

Let $(E,\pi:E\to[0,t_1])$ be the pull-back of the tautological bundle so that 
  each fiber $\pi^{-1}(t)$ is a copy of $q_t$ for each $t\in[0,t_1]$.  
Let $\Omega\subset E$ be the subset that intersects each fiber in $K_t$.  
Define $K_{t_1}=\Bar\Omega\setminus\Omega$ and note that it is a closed 
  set contained in the fiber over $t=t_1$.  
Let $f_{t_1}:K\to K_{t_1}$ be the pointwise limit of the maps $f_t$ as $t\to t_1^-$.  
Each saddle connection $\gamma$ of $\Gamma$ is mapped by $f_{t_1}$ to 
  a union of parallel saddle connections $\gamma'_i,i=1,\dots,r=r(\gamma)$.  
A triangle $\Delta$ determined by $\Gamma$ may collapse under $f_{t_1}$ 
  to a union of parallel saddle connections; otherwise, $f_{t_1}(\Delta)$ has 
  $n=n(\Delta)$ saddle connections on its boundary, possibly with $n>3$.  
If $n>3$, then $n-3$ zeroes of $q_t$   hit the interior of an edge of $\Delta$ at $t=t_1$.
 In this case we triangulate $f_{t_1}(\Delta)$ by adding $n-3$ "extra" 
  saddle connections.  
Let $\Gamma_{t_1}$ be the collection of saddle connections $\gamma'_i$ 
  associated to $\gamma\in\Gamma$ together with the ``extra" saddle 
  connections needed to triangulate $f_{t_1}(\Delta)$ for $\Delta$ that do 
  not collapse and have $n(\Delta)>3$.  
Let $K'_{t_1}$ be the complex determined by $\Gamma_{t_1}$ and let $f'_{t_1}:K\to K_{t_1}'$ be the composition of $f_{t_1}$ with the inclusion of $K_{t_1}$  into $K'_{t_1}$.  
It is easy to see that $F_{t_1}:=f'_{t_1}$ maps saddle connections to unions of saddle connections and that $K'_{t_1}$ does not depend on the choice of $\Gamma$.

If $t_1=1$ we are done and we set $K'=K'_{t_1}$.
Thus assume  $t_1<1$. 
 We repeat the construction above starting with $K'_{t_1}$ and form the maximal set $A(t_1)$ of times $t_1\leq t<t_2$   such that the homotopy class of each saddle connection of $K'_{t_1}$ is realized by a single saddle connection on $q_t$.  We repeat the procedure, building a new complex $K'_{t_2}$ and finding a map $f'_2:K'_{t_1}\to K'_{t_2}$.  We then let $F_{t_2}=F_1\circ f'_{t_1}$.  The compactness of $[0,1]$ implies that this procedure only need be repeated a finite number of times $t_1<t_2<\dots<t_N=1$.  We inductively find $F_{t_N}$ and set $F=F_{t_N}$ and $K'=K'_{t_N}$. 
\end{proof}

\begin{definition}
Let $q'$ be an $\eps$-perturbation of $q$ and $K$ a complex in $q$.  
Let $K'$ be the complex obtained by applying Theorem~\ref{thm:move} using 
the linear path in the stratum joining $q$ and $q'$.  
We call $K'$ the {\em moved} complex.  
\end{definition}

\begin{cor} \label{cor:move:ang} Let $q'$ be an $\eps$-perturbation of $q$ and suppose that $K$ is a complex on $q$ that moves to $K'$ on
 $q'$ by $F$. Let $\gamma \in K$. Let $\cup_1^k \gamma_i'=F(\gamma)$. Then $|\theta_{\gamma}-\theta_{\gamma_i'}|<2\lambda_2\eps$.
\end{cor}
\begin{proof}
Each $f'_{t_k}, i=1,\dots,N$ in the proof of Theorem~\ref{thm:move} is a piecewise 
linear map between complexes on $q_k$ that are $\eps_k$-perturbations of one another, where $\sum_{k=1}^N \eps_k=\eps$.  Let $\sigma_k,k=0,1,\dots,N$ be the saddle connections such that $\sigma_0=\gamma$, $f'_{t_k}(\sigma_{k-1})\supset\sigma_k$, 
and $\sigma_N=\gamma'_i$.  Lemma~\ref{lem:angle} implies $|\theta_{k-1}-\theta_k|<2\lambda_2\eps_k$ where $\theta_k$ denotes the direction of $\sigma_k$.  The conclusion of the corollary now follows from the triangle inequality.  
\end{proof}

\begin{prop}\label{prop:inclusion}
Suppose $K_1\subset K_2$ as closed subsets and each is a complexes on $q_0$. They are both moved to  $q_1$ to become complexes $K_1',K_2'$. Then again viewed as closed subsets, we have $K_1'\subset K_2'$. 
\end{prop}

\begin{proof}
We can extend the triangulation of $K_1$ to a triangulation $\hat K_1$ of the same closed subset as is defined by $K_2$ and such that $\hat K_1$ and $K_2$ coincide on the boundary.  
We move both $K_2$ and $\hat K_1$ to $q_1$ obtaining triangulations $K_2'$ and $\hat K_1'$.  Theorem~\ref{thm:move} says that $K_2'$ and $\hat K_1'$ are triangulations of the same closed set. 
Clearly $K_1'\subset \hat K_1'$ and we are done. 
\end{proof}

We adopt the notation $(K,q)$ to refer to a complex on the flat surface defined by $q$.

\begin{definition}
Suppose $(K_1,q_1)$ and $(K_2,q_2)$ are complexes of distinct flat surfaces.  We say  $(K_1,q_1)$  and  $(K_2,q_2)$ are \emph{not combinable} if  $K_1$ moved to $q_2$ satisfies   $K_1'\subseteq K_2$ and $K_2$ moved to $q_1$ satisfies $K_2'\subseteq K_1$. Otherwise they are said to be \emph{combinable}. 
\end{definition}

\subsection{Proof of Theorem \ref{thm:fullspace}}
We are given the compact set $\bar U$ in the stratum. For any $\alpha,\beta$ we can play the game a finite number of steps so that we are allowed to assume  that $U$ is a ball with center $q_0$ which has a  triangulation $\Gamma$ which remains defined   for all $q\in U$.  We are  therefore able to talk about $\epsilon$ perturbations in $U$.  Furthermore since $\bar U$ is compact, the constant  $\lambda_2$
given by Lemma~\ref{lem:comparable} which depends only on the  and the triangulation can be taken to be uniform in $U$. 
Furthermore because our choice of metric is the sup metric, each ball $B_j$ will be of the form $$B_j=Z_j\times I_j,$$
where $Z_j$ is a ball in the space $Z$ and $I_j\subset S^1$.

Now choose 
\begin{equation}
\label{eq:alpha}
\alpha<\min\{\frac 1 8,\frac{1}{720(6g-6+n)^2\lambda_2}\},
\end{equation}
where $n$ is the number of zeroes.  (Again the significance of the choice of constants $8$, $720$ is only to make certain inequalities hold.) 
Let $L_0$ denote the length of the shortest saddle connection on $q_0$.  
Let $0<c_1<\dots<c_{M+1}$ be given by 
  $c_i=(\alpha\beta)^{N_i}c_{i+1}$ and 
  $$c_{M+1} = \min( L_0\beta^{N_M},L_0|I_1|^{1/2}\eps_0 )$$ 
  where $N_i$ are defined by $N_1=4M+1$ and 
  $N_{i+1} = 4M + 4(N_1+\dots+N_i)$ so that 
\begin{equation}\label{def:c_i:2}
  \frac{c_{i+1}}{c_i} = (\alpha\beta)^{-4M-1}\left(\frac{c_i}{c_1}\right)^4 
          \ge 100(\alpha\beta)^{-4M}\left(\frac{c_i}{c_1}\right)^4.  
\end{equation}

Now inductively, given a ball $B_j=Z_j\times I_j$, where $Z_j$ is centered at  $q_j$, 
let $\cE_i(B_j):=\cE_i$ be the set of all marked $(\alpha\beta)^{3+M} c_i$-shrinkable complexes $(K,q)$ of level $i$ where  $q\in B_j$.  
Given $I_j$ with $j \equiv M-i+1$ mod $M$, we let 
\begin{equation*}
  \cA_i(j) := \left\{(K,q)\in\cE_i : d(\theta(K),I_j) > \frac{(\alpha\beta)^M c_i^2}{24L(K)L(\partial K)} \right\}.
\end{equation*}

We shall prove the following statement for every $j\ge1$ and every  $i\in\{1,\dots,M\}$, where $j \equiv M-i+1$ mod $M$. 
\begin{align*}
\tag{$P_j$} \; \forall K\in\cE_i \qquad
                                 L(K)L(\partial K)|I_j| < c_i^2 \;\;\Longrightarrow\;\; K\in \cA_i(j)
\end{align*}

Note that ($P_j$)  holds automatically for all $j<j_0:=\min\{k: |I_k|<\beta^{2N_M}\}$ because 
  $$L(K)L(\partial K)|I_j| \ge L_0^2\beta^{2N_M} \ge c_{M+1}^2 > c_i^2$$ 
  while if $j_0=1$, we note that $L(K)L(\partial K)|I_1|\ge L_0^2|I_1| \ge c_{M+1}^2>c_i^2.$  
  
We proceed by induction and suppose that Alice is given a ball $B_j=Z_j\times I_j$, $j\ge j_0$, where $Z_j\subset Z$ is a ball  and $I_j\subset I$ is an interval. Suppose inductively  that ($P_{k}$) holds for $k\leq j$. We will show that Alice has a choice of a ball $A_j\subset B_j$ to ensure $(P_{j+M})$ will  hold.   

Define
 \begin{equation*}
  \Omega_i(j) := \left\{ (K,q)\in\cE_i\setminus\cA_i(j) : 
     \frac{c_i^2}{L(K)L(\partial K)} \le |I_j|< \frac{(\alpha\beta)^{-M}c_i^2}{L(K)L(\partial K)} \right\}. 
\end{equation*}

We summarize our strategy.   We will show that for any $q_1,q_2 \in B_j$  such that $d_U(q_1,q_2)\leq \alpha |B_j|$, no two 
  complexes $(K_1,q), (K_2,q)$  are combinable (step 2). Then we will show that if $d_U(q',q)\leq \alpha|Z_j|$, and $(K,q),(K',q')\in\Omega_i(j)$ are complexes
  which are not pairwise combinable, then $|\theta(K)-\theta(K')|$ is small (step 3). Then choosing some $(K,q)$,  Alice can choose an angle $\phi$ where  $d(\phi,\theta(K))>\frac 1 3 |I_j|$,  and an  interval $I_j'$ centered at $\phi$  of radius  $\alpha|I_j|$. 
 There will be no $(K',q')$ with  $\theta(K')\in I_j'$ and where $K'$ is combinable with $K$ and  $d_U(q,q')\leq  \alpha|Z_j|$ (step 4). As in the proof of Theorem \ref{blocking strategy}, step 1 is a technical result controlling the ratio of $L(K)$ and $L(\partial K)$, which is necessary for step 2.

\textit{Step 1.}  
We show that for any $(K,q)\in\Omega_i(j)$, 
\begin{equation}\label{ieq:boundary}
  \frac{L(K)}{L(\partial K)} < (\alpha\beta)^{-2M}\left(\frac{c_i}{c_1}\right)^2.  
\end{equation}

This is essentially the same as the proof of Step 1 in the proof of Theorem~\ref{blocking strategy}. We provide the details.
Let $k$ be the previous stage for dealing with 1-complexes. Consider first the case that \begin{equation}
\label{eq:caseA}
L(\partial K)^2|I_j| <(\alpha \beta)^M c_1^2. 
\end{equation}
Then $L(\partial K)^2|I_{k}| < c_1^2$ so that ($P_{k}$) implies the longest saddle 
connection on $\partial K$ belongs to $A_1(k)$, meaning 
\begin{equation*}
  d(\theta(\partial K),I_{k}) > \frac{(\alpha\beta)^M c_1^2}{24L(\partial K)^2}>c_1^2|I_k|.  
\end{equation*}
But since $K$ is $(\alpha\beta)^{3+M} c_i$-shrinkable and not in $\cA_i(j)$, the triangle inequality implies 
\begin{align*}
  d(\theta(\partial K),I_{j}) &\le \left| \theta(\partial K) - \theta(K) \right| + d(\theta(K),I_j) \\
             &\le \frac{(\alpha\beta)^{6+2M}c_i^2}{L(K)L(\partial K)} + \frac{(\alpha\beta)^M c_i^2}{24L(K)L(\partial K)} 
             \le \frac{(\alpha\beta)^M c_i^2}{L(K)L(\partial K)}
\end{align*}
which implies (\ref{ieq:boundary}) using the fact that $\alpha\beta<1$.  Suppose now that (\ref{eq:caseA}) does not hold. Since
$(K,q) \in \Omega_i(j)$, we have 
\begin{equation*}
  \frac{L(K)}{L(\partial K)} = \frac{L(K)L(\partial K)|I_j|}{L(\partial K)^2|I_j|} 
                                             < \frac{(\alpha\beta)^{-M}c_i^2}{(\alpha\beta)^M c_1^2}
\end{equation*}
so that (\ref{ieq:boundary}) holds in this   case as well.

\textit{Step 2}. Now we show that if $(K_1,q_1)$ and $(K_2,q_2)$ are in $\Omega_i(j)$ and $$d_U(q_1,q_2)\leq \alpha |B_j|,$$ then $ K_1$ and   $K_2$ are not combinable. 
Assume on the contrary that they are combinable.  So without loss of generality assume there exists $\gamma \in K_2$  so that the moved  $\gamma' \not \subset K_1$.

Choose the following constants. 

Let $\eps=c_i(\alpha\beta)^{3+M}, \qquad \rho_1=2(\alpha\beta)^{-4M}\left(\frac{c_i}{c_1}\right)^2, 
               \quad \rho_2=(\alpha\beta)^{-2M}\left(\frac{c_i}{c_1}\right)$ and  \quad $\rho_3=12\rho_2(\rho_1+\rho_2)$.
               
We now show that we can combine $\gamma'$ to $K_1$ to make a $\sqrt{\rho_3}\epsilon$-shrinkable complex.
Observe that similarly to the proof of Step 2 in  Theorem~\ref{blocking strategy}  we have 
$$|\theta(K_1)-\theta(K_2)|\leq (1+\frac 1 {12})|B_j|\leq \frac{\rho_1\epsilon^2}{L(K_1)L(K_2)} \text{ and } \frac{L(K_1)}{L(K_2)}\leq \rho_2.$$

Therefore
\begin{multline}
|\theta_{\gamma'}-\theta(K_1)|
\leq |\theta_{\gamma}-\theta_{\gamma'}|+
|\theta_{\gamma}-\theta(K_2)|+|\theta(K_1)-\theta(K_2)|
\leq \\ \lambda_2 \|q_1-q_2\|+
 \frac{3\epsilon^2}{2|\gamma'|L(K_2)}+\frac{\rho_1\epsilon^2}{L(K_1)L(K_2)}
\leq  \frac{2\epsilon^2}{|\gamma'|L(K_2)}+\frac{2\rho_1\epsilon^2}{L(K_1)L(K_2)}\leq \frac{(2\rho_1+2\rho_2)\epsilon^2}{|\gamma'|L(K_1)}
\end{multline} 

Let $\theta=\theta(K_1)$ and $t=\log\frac{(L(K_1)}{\epsilon}$.
There is a saddle connection $\sigma$ disjoint from  $K_1$ such that on  $g_tr_\theta X$ the saddle connection $\sigma_{\theta,t}=g_tr_\theta\sigma$ satisfies

$$h_{\theta}(\sigma_{\theta,t})\leq (2\rho_2+2\rho_1+3)\epsilon \text{ and } v_{\theta}(\sigma_{\theta,t})\leq  \frac{3}{2}\rho_2\epsilon .$$
It follows that 
$$|\sigma |\leq \sqrt{\left(\frac{(2\rho_2+2\rho_1+3)\epsilon}{L(K_1)}\right)^2+\left(\frac{3\rho_2L(K_1)}{2}\right)^2}\leq 2\rho_2 L(K_1).$$

Now we show $\hat{K}=K_1\cup \sigma$ is $\epsilon\sqrt{\rho_3}$ shrinkable.
We have $L(\hat{K})=\max\{L(K_1),|\sigma|\}$.
If $L(\hat{K})=L(K_1)$ then $\hat{K}$ is $\epsilon\sqrt{\rho_3}$ shrinkable because

$$|\theta_\sigma-\theta(K_1)|\leq 2\frac{h_{\theta}(\sigma)}{|\sigma|} \leq  \frac{(5\rho_1+5\rho_2)\epsilon^2}{|\sigma|L(K_1)}.$$
If $L(\hat{K})=|\sigma|$ 
 for every $\xi\in K_1$ we have
\begin{align*}
  \left| \theta(\hat{K})-\theta_\xi \right| 
      &\le \left| \theta_\sigma - \theta(K_1) \right| + \left| \theta(K_1) - \theta_\xi \right| \\
      &< \frac{(5\rho_1+5\rho_2)\eps^2}{|\sigma| L(K_1)} + \frac{\eps^2}{|\xi| L(K_1)} 
      < \frac{(6\rho_1+6\rho_2)\eps^2}{|\xi| L(K_1)} 
      < \frac{(12\rho_1+12\rho_2)\rho_2\eps^2}{|\xi| L(\hat{K})}
\end{align*}
  where $2\rho_2L(K_1)\ge|\sigma|\ge L(K_1)\ge|\xi|$  was used in last two inequalities. 
   This shows that $\hat{K}$ is $\sqrt{\rho_3}\epsilon$-shrinkable.

   Now we derive a contradiction. Notice that $c_{i+1}(\alpha\beta)^{3+M}>\sqrt{\rho_3}\epsilon$ and so $\hat{K}$ is $c_{i+1}(\alpha\beta)^{3+M}$ shrinkable.
   
   Since there are no $c_{M+1}$-shrinkable complexes of level $M+1$, we have our 
  desired contradiction when $i=M$.  For $i<M$, we have 

\begin{align*}
  L(\hat{K})L(\partial \hat{K})|B_{j-1}| < \frac{4\rho_2^2L(K_1)^2|B_j|}{\alpha\beta}
                                          < \frac{4\rho_2^2c_i^2L(K_1)}{(\alpha\beta)^{1+M}L(\partial K_1)}
                                          < 4(\alpha\beta)^{-7M-1}c_i^2\left(\frac{c_i}{c_1}\right)^4 < c_{i+1}^2.
\end{align*}
The first inequality uses the bound on $L(\hat K)$ in terms of $L(K_1)$ and the definition of the game. The third inequality uses Step 1.  We conclude that   the induction hypothesis ($P_{j-1}$) implies $\hat{K}\in\cA_{i+1}(j-1)$, meaning 
  \begin{equation}
\label{eq:contradiction}
d_a(\theta(\hat{K}),I_{j-1}) > \frac{(\alpha\beta)^M c_{i+1}^2}{24L(\hat{K})L(\partial \hat{K})} 
                                     \ge \frac{(\alpha\beta)^M c_{i+1}^2}{24L(\hat{K})^2}.
\end{equation}  Since $\hat K$ is $\sqrt{\rho_3}\epsilon$ shrinkable by the  choice of $\epsilon$, it is  in fact ${c_i \sqrt{\rho_3}(\alpha\beta)^M}$-shrinkable, 
  and since ${c_{i+1}>c_i\sqrt{\rho_3}\sqrt{4(\alpha \beta)^{-M} \rho_2}}$, we have 
\begin{multline*}
  d_a(\theta(\hat K),I_{j-1}) \le \left| \theta(\hat K) - \theta(K_1) \right| + d_a(\theta(K_1),I_j) 
  \leq \frac{\rho_3 (\alpha\beta)^Mc_i^2}{L(K_1)L(\hat{K})}+\frac{(\alpha\beta)^M}{24L(K_1)L(\partial K_1)} \\
  <\frac{2\rho_3 \rho_2(\alpha\beta)^Mc_i^2}{L(\hat{K})^2}+\frac{(\alpha\beta)^M2\rho_2\rho_2\rho_1c_i^2}{L(\hat{K})^2}
       <  \frac{(\alpha\beta)^M c_{i+1}^2}{24L(\hat{K})^2}
\end{multline*}
  giving us the desired contradiction to (\ref{eq:contradiction}).

\textit{Step 3}. 
We show that if $(K_1,q_1)$ is not combinable with $(K_2,q_2)$, each belongs to $\Omega_i(j)$,  and $d_U(q_1,q_2)\leq 2 \alpha |B_j|$, then $|\theta(K_1)-\theta(K_2)|\leq \frac 1 3 |B_j|$.
\\

Since $(K_1,q_1)$ and $(K_2,q_2)$ are not combinable, when we move  $K_1$  to $q_2$ which we denote by $K_1'$, we have $K_1'\subseteq K_2$.  Similarly $K_2'\subset K_1$. 
By Proposition~\ref{prop:inclusion}
when we move  $K_1'$ back to $q_1$, denoted by $K_1''$, we have $$K_1''\subset K_2'\subset K_1.$$ But since each saddle connection on $\partial K_1$ is homotopic to a  union of saddle connections of $K_1''$ with common endpoints, 
$K_1$ and $K_1''$ bound a union of   (possibly degenerate)  simply connected domains each of which has a segment of $\partial K_1$ as a side.   Let  $\gamma$ be the longest saddle connection on $\partial K_1$ and let $\Delta$ be the corresponding simply connected domain.  Since $K_1''\subset K_2'\subset K_1$,  $\Delta$ contains a union of saddle connections $\hat\kappa'\subset \partial K_2'$
that join the endpoints of $\gamma$.  Let $p\leq 6g-6+n$ denote the cardinality of $\hat \kappa'$.

Since $d_U(q_1,q_2)\leq 2\alpha |B_j|$, each is a $2\alpha |B_j|$ perturbation of the other. 
Since the angle the saddle connections of $\partial K_2$ make with each other goes to $\pi$ as the length of the segments goes to $\infty$, and these angles change by a small factor, by Corollary~\ref{cor:move:ang} we have for some $\kappa'\in\hat\kappa'$,
\begin{equation}\label{compare2}L(\partial K_1)\leq 2p|\kappa'|.
\end{equation}

and for all  $\kappa'\in\hat\kappa'$,
\begin{equation}
 \label{compare1}  |\kappa'|\leq pL(\partial K_1).
\end{equation}

Since lengths change by a factor of at most $1+\lambda_1\leq \frac{3}{2}$ in moving, and since $\kappa'$ arises from a saddle connection $\kappa\subset \partial K_2$, we have that $$L(\partial K_1)\leq 3p|\kappa|\leq 3pL(\partial K_2),$$ and by symmetry $$L(\partial K_2)\leq 3rL(\partial K_1),$$ for some constant $r\leq 6g-6+n$,  so that 
\begin{equation}
\label{eq:kappa}|\kappa'|\geq \frac{L(\partial K_2)}{9rp}\geq \frac{L(\partial K_2)}{9(6g-6+3n)^2}.
\end{equation}

We also claim that 
 $$|\theta_{\kappa'}-\theta(K_1)|<\frac 1 {20} |B_j|.$$
 To see this, by Corollary~\ref{cor:move:ang} applied twice, first  to the moved $K_1'$ and then to $K_1''$, and by the choice of $\alpha$, in  (\ref{eq:alpha}),   we have that for all $\gamma'\in \partial \Delta$ 
 $$|\theta_{\gamma'}-\theta_\gamma|\leq \frac{1}{40}|B_j|,$$ 
which implies since $\kappa'$ is a  union of saddle connections in  $\Delta$ that $$|\theta_\gamma-\theta_{\kappa'}|\leq \frac{1}{40}|B_j|.$$ 
 The shrinkability of $K_1$ implies that 
$$|\theta(K_1)-\theta_\gamma|\leq  \frac{c_i^2(\alpha\beta)^{2M}}{L(K_1)|\gamma|}\leq \frac{1}{40}|B_j|,$$ so the claim follows from the last two inequalities.

By Corollary~\ref{cor:move:ang},  the triangle inequality, and  the above claim  we have
 $$ |\theta(K_1)-\theta(K_2)|\leq |\theta(K_1)-\theta_{\kappa'}|+|\theta_{\kappa'}-\theta_\kappa|+|\theta_\kappa-\theta(K_2)|\leq \frac{1}{20}|B_j|+\frac{1}{40}|B_j|+|\theta(K_2)-\theta_\kappa|.$$
  By our shrinkability assumption on $K_2$,  
$$|\theta_\kappa-\theta(K_2)|\leq \frac{c_i^2(\alpha\beta)^{2M}}{L(K_2)|\kappa|}\leq\frac{1}{20}|B_j|,$$
where the second  inequality follows from 
 (\ref{eq:kappa}), the definition of $\Omega_i(j)$,  and the choice of $\alpha$ given in (\ref{eq:alpha}).  Step 3 follows. 

\textit{Step 4}.

Bob presents Alice with a ball $B_j=Z_j\times I_j$ where $j \equiv M-i+1 $ mod $M$. If there is no $(K,q)$ in $\Omega_i(j)$, Alice makes an arbitrary move. 
Otherwise, pick a $(K,q)\in \Omega_i(j)$. Alice chooses a ball ${A_j=Z_j'\times I_j' }$ of diameter $\alpha|B_j|$, 
 whose center has first coordinate $q$ and second coordinate is as far from $\theta(K,q)$ as possible. Observe that $${d_a(\theta(K,q),I_j)'\geq( \frac 1 2 -2\alpha)|B_j|}.$$
 By Step 2 if $\hat{q}\in Z_j$ and $(\hat{K},\hat{q})\in \Omega_i(j)$ then $\hat{K}$
 and  $K$ are not combinable.   By Step 3
  $$d(\theta(\hat{K
},\hat q),I_j')\geq (\frac 1 2-2\alpha) |B_j|-\frac 1 3 |B_j|\geq (\frac 1 6 - 2\alpha)|B_j|.$$  Now because $(\hat{K}, \hat q) \in \Omega_i(j)$, using the left hand inequality in the definition and that $\alpha\beta<1$,  we conclude  that $$d_a(\theta(\hat{K}, \hat q),I_j')\geq (\frac 1 6-2\alpha)\frac{c_i^2(\alpha\beta)^M}{L(\hat{K})L(\partial \hat{K})}.$$ Because $d_a(\theta,I_j')\leq d_a(\theta,I_{j+M}')$ we know then that $(P_{j+M})$ holds.  This finishes the inductive proof of $(P_j)$.


We finish the proof of Theorem~ \ref{thm:fullspace}.  By $(P_j)$ we are able to ensure that  for any level $i$ complex $K$ on a surface $q$, we have 
$$ \max\{L(\partial K) \cdot L(K) \cdot |B_{j}|, L(\partial K) \cdot L(K) \cdot d_\theta((q,K),B_{j}\}> \frac{(\alpha\beta)^{M} c_i^2}{4}.$$
  In particular this holds when $i=1$.  Since there is only one saddle connection in a 1-complex,  and since  for any fixed  saddle connection $\gamma$ on a surface $q$, $|\gamma|^2|B_{j}|\to 0$ as $j\to\infty$,  we conclude that  for all but finitely many balls $B_{j}$  we have
$$|\gamma|^2 d_\theta((q,\gamma),B_{j})=\max\{|\gamma|^2|B_{j}|, |\gamma|^2 d_\theta((q,\gamma), B_{j}\}>\frac{(\alpha\beta)^{M} c_1^2}{4}.$$ 
Thus if $(q,\phi)=\cap_{l=-1}^\infty B_{l}$ is the point we are left with at the end of the game, and $\gamma$ is a saddle connection on $q$, then 
$|\gamma|^2|\theta_{\gamma}-\phi|>\frac{(\alpha\beta)^{M} c_1^2}{4}$, which by Proposition \ref{equiv cond} establishes the statement of strong winning in Theorem \ref{thm:fullspace}.  

The set cannot be absolute winning for the following reason. 
Bob begins by choosing a ball $I_1$ centered at some quadratic differential which has a vertical saddle connection $\gamma$. The set  $X_\gamma$ consisting of quadratic differentials with a vertical saddle connection $\gamma$ is a closed subset of codimension one and such quadratic differentials  are clearly not bounded.  Then  whatever Alice's move of a ball $J_1\subset I_1$,  Bob  can find a next ball $I_2\subset I_1\setminus J_1$ centered at some new point in  $X_\gamma$ which shows that bounded quadratic differentials  are not absolute winning.

An identical proof allows us the same theorem in the case of marked points. 

\begin{thm}\label{marked strat} Let $Q$ be a stratum of quadratic differential with $k$ marked points. Let $U \subset Q$ be an open set with compact closure in $Q$ where the metric given by local coordinates is well defined. The set  $E\subset \bar U$ consisting of those quadratic differentials $q$ such that the Teichm\"uller geodesic defined by $q$ stays in a compact set in the stratum  is $\alpha$  winning for Schmidt's game.   In fact it is $\alpha$-strong winning.
\end{thm}

In fact with a similar proof we have the following 
\begin{thm}\label{thm:rotinvar} Let $P$ be a  rotation invariant subset of the stratum of quadratic differentials with $k$ marked points where a metric given by local coordinates is well defined. Assume $P$ has compact closure in the stratum.  The set $E\subset P$ consisting of those quadratic differentials $q$ such that the Teichm\"uller geodesic defined by $q$ stays in a compact set in the stratum  is $\alpha$  winning for Schmidt's game.   In fact it is $\alpha$-strong winning. 
\end{thm}


\subsection{Proof of Theorem~\ref{thm:pmf}}\label{sec:pmf}
Again as before the Diophantine foliations are not absolute winning since 
for any closed curve $\gamma$ the set of foliations $F$ such that 
$i(F,\gamma)=0$ is a codimension one subset. 

We now show strong winning.  We can assume we start with a fixed train track $\tau$, 
and a small ball $B(F_0,r)$ of foliations carried by $\tau$.  Indeed, if the initial ball 
Bob chooses contains points on the boundary of two or more charts, then Alice can 
use the strategy of choosing her balls furthest away from these boundary points, 
so that in a finite number of steps her choice will be contained in a single chart.  

By choosing transverse foliations, we can insure that there a ball $B'(q_0,r')\subset Q^1(1,\ldots, 1,-)$ of quadratic differentials contained in the principal stratum so that  
\begin{itemize}
\item the vertical foliation of each $q\in B'(q_0,r')$ is in $B(F_0,r)$. 
\item each vertical foliation in $B(F_0,r)$ is the vertical foliation of  some $q\in B(q_0,r')$.  
\item  $B(F_0,r)$ and $B'(q_0,r')$  are  small enough so that the holonomies of a fixed set of saddle connections serve as local coordinates.
\item  There is a fixed constant so that the holonomy   of any $q\in B'(q_0,r')$ is  bounded away from $0$ by that constant.
\end{itemize}
  In holonomy coordinates the map that sends $q\in B'(q_0,r')$ to its vertical foliation is just projection onto the horizontal coordinates. This map clearly satisfies the hypotheses of Theorem~\ref{abstract thm}.  Since the bounded geodesics form a strong winning set in $Q^1(1,\ldots, 1,-)$ by Theorem~\ref{thm:fullspace} they are strong winning in $\PMF$.

\subsection{Proof of Theorem~\ref{thm:iet}}

If the 
condition $\inf_n n|T^n(p_1)-p_2|>0$ holds for any pair of discontinuities $p_1,p_2$ of $T$, we say $T$ is {\em badly approximable}. 
The following lemma connects the badly approximated condition for interval exchanges with the bounded condition for geodesics. 

\begin{lemma}\label{ba is ba}(Boshernitzan \cite[Pages 748-750]{dukeJ}) $T$ is badly approximable if and  only if  the Teichm\"uller geodesic corresponding to vertical direction is bounded for any  zippered rectangle such that  $T$ arises as the first return of the vertical flow to a transversal. 
\end{lemma}
See in particular the first equation on page 750, which relates the size of smallest interval bounded by discontinuities of $T^n$ and closeness to a saddle connection direction. That is, let $T$ be an IET that arises from first return to a transversal of a flow on a flat surface $q$. Assume that the smallest interval of continuity of $T^n$  is less than $\frac{\epsilon}{n}$ 
then $$ |\theta(v)-\theta|<\frac{C\epsilon}{nL(v)},$$ where $v$ is a saddle connection on $q$ with length $O(n)$.
Recall 
Theorem \ref{thm:equiv} relates closeness to saddle connection directions to boundedness of the Teichm\"uller geodesics.

We now give the proof  of Theorem~\ref{thm:iet}.
For the same reason as above the set of bounded interval exchanges is not absolute winning.  
For strong winning, the proof is identical to the one for $\PMF$ except  that now, using for example the  the zippered rectangle construction,  we can assume we have a small ball $B$ in the space of interval exchange transformations \cite{gauss}, a corresponding ball in some stratum $B'\subset Q^1(k_1,\ldots, k_n,+)$, such that each  interval exchange transformation in $B$ arises from the first return to  a horizontal transversal of some $\omega\in B'$, and conversely for each $\omega\in B'$, the first return to a horizontal transversal gives rise to a point in $B$.  We can assume these transversals vary continuously.  Again the map from holonomy coordinates in $B'$ to lengths in $B$ is given by projection onto horizontal coordinates.  We now apply Lemma~\ref{ba is ba}, 
 Theorem \ref{thm:fullspace} and Theorem~\ref{abstract thm}.

If we mark points $a,b$ in the interval then we have  

\begin{thm}\label{marked} Given any irreducible permutation $\pi$ there exists $\alpha>0$ such that for any pair of points $(a,b)$ we have
$$\{T=T_{L,\pi}: \inf_{n>0}\{nd(T^na,b)\}>0\}$$
 is an $\alpha$-strong winning set. 
\end{thm}

\begin{proof}   As in the last theorem we find a ball in the stratum such that first return to transversals give the interval exchange. Now mark the points along each transversal at distances $a,b$ to obtain a  set  $B'$ of marked translation surfaces. It is not a ball but it is invariant under rotations lying in a small interval about the identity. 
Then by Theorem~\ref{thm:rotinvar} the set of bounded trajectories in it is strong winning. 
 By Theorem \ref{abstract thm} the image of this set is strong winning in the space of marked interval exchange transformations. 
\end{proof}


\subsection{Proof of Theorem~\ref{thm:riemann}}
Let $U$ be the intersection of the principle stratum with $Q^1(X)$.  Since the complement of $U$ is contained in a finite union of smooth submanifolds, then for any sufficiently small $\alpha>0$ and for any sufficiently small ball chosen by Bob, Alice can respond with a ball contained entirely in $U$ with the  bounded away from zero.  Thus, we may assume Bob's initial ball $B_1$ is contained in $U$.  By the main theorem of \cite{HM}, the homeomorphism from $Q^1(X)\to\PMF$ sending a quadratic differential to the projective class of its vertical foliation is smooth when restricted to $U$.  We can now apply Theorem~\ref{thm:pmf}.

\end{document}